\documentclass[10pt]{amsart}
\usepackage{latexsym}
\usepackage{amscd, amsfonts, eucal, mathrsfs, amsmath, amssymb, amsthm, stmaryrd}
\usepackage{fontenc}
\usepackage{pdfsync}
\usepackage{newcent}

\input xy
\xyoption{all}

\usepackage{stmaryrd}
\SetSymbolFont{stmry}{bold}{U}{stmry}{m}{n}

\usepackage{fullpage} 

\newcommand{\field}[1]{\mathbf #1}
\newcommand{\mf}[1]{\mathfrak #1}
\newcommand{\mc}[1]{\mathcal #1}
\newcommand{\ms}[1]{\mathscr #1}
\newcommand{\widebar}[1]{\overline{#1}}

\newcommand{\R}{\field R}
\newcommand{\LL}{\field L}
\newcommand{\C}{\field C}
\newcommand{\F}{\field F}
\newcommand{\Z}{\field Z}

\newcommand{\V}{\field V}

\newcommand{\simto}{\stackrel{\sim}{\to}}

\renewcommand{\phi}{\varphi}

\newcommand{\unif}{\text{\rm unif}}
\newcommand{\sm}{\text{\rm sm}}

\renewcommand{\hom}{\operatorname{Hom}}

\newcommand{\shom}{\ms H\!om}

\newcommand{\rshom}{\mathbf{R}\shom}

\newcommand{\send}{\ms E\!nd}

\newcommand{\spec}{\operatorname{Spec}}
\newcommand{\Spec}{\operatorname{Spec}}

\renewcommand{\P}{\field P}

\DeclareMathOperator{\Pic}{Pic}

\DeclareMathOperator{\supp}{Supp}
\DeclareMathOperator{\pr}{pr}

\newcommand{\m}{\boldsymbol{\mu}}

\newcommand{\G}{\field G} 

\DeclareMathOperator{\Def}{Def}

\renewcommand{\H}{\operatorname{H}}

\newcommand{\GL}{\operatorname{GL}}

\DeclareMathOperator{\ext}{\operatorname{Ext}}

\newcommand{\Gal}{\operatorname{Gal}}

\DeclareMathOperator{\per}{per}
\DeclareMathOperator{\ind}{ind}

\renewcommand{\)}{)\!\hspace{0.03em})}
\DeclareMathOperator*{\tensor}{\otimes}

\DeclareMathOperator{\Tr}{\operatorname{Tr}}
\DeclareMathOperator{\adj}{\operatorname{adj}}

\newcommand{\surj}{\twoheadrightarrow}
\newcommand{\inj}{\hookrightarrow}

\DeclareMathOperator{\coker}{\operatorname{coker}}

\DeclareMathOperator{\End}{\operatorname{End}}
\DeclareMathOperator{\aut}{\operatorname{Aut}}
\DeclareMathOperator{\Aut}{\operatorname{Aut}}
\DeclareMathOperator{\isom}{\operatorname{Isom}}

\DeclareMathOperator{\M}{\operatorname{M}}
\DeclareMathOperator{\Br}{\operatorname{Br}}
\DeclareMathOperator{\Bl}{\operatorname{Bl}}

\DeclareMathOperator{\Sing}{Sing}

\DeclareMathOperator{\B}{\operatorname{\mathsf B\!}}

\newtheorem{lem}{Lemma}[subsubsection]

\makeatletter
\@addtoreset{lem}{section}
\@addtoreset{lem}{subsection}
\@addtoreset{lem}{subsubsection}
\@addtoreset{lem}{paragraph}
\makeatother

\renewcommand{\thelem}{\ifnum\value{subsubsection}>0{\thesubsubsection.\arabic{lem}}\else{\ifnum\value{subsection}>0{\thesubsection.\arabic{lem}}\else{\thesection.\arabic{lem}}\fi}\fi}

\newtheorem{thm}[lem]{Theorem}

\newtheorem{prop}[lem]{Proposition}

\newtheorem{cor}[lem]{Corollary}

\newtheorem*{claim}{Claim}

\theoremstyle{definition}
\newtheorem{defn}[lem]{Definition}

\newtheorem{example}[lem]{Example}
\newtheorem{hyp}[lem]{Hypothesis}

\newtheorem{para}[lem]{}

\newtheorem{construction}[lem]{Construction}
\newtheorem{notn}[lem]{Notation}
\newtheorem{notation}[lem]{Notation}
\newtheorem{goal}[lem]{Goal}

\theoremstyle{remark}
\newtheorem{remark}[lem]{Remark}

\numberwithin{equation}{lem}

\author{Max Lieblich}
\address{Padelford Hall, Seattle}
\email{lieblich@math.washington.edu}
\title{The period-index problem for fields of transcendence degree $2$}

\begin{document}
\begin{abstract}
  Using geometric methods we prove the standard period-index
  conjecture for the Brauer group of a field of transcendence
  degree $2$ over $\F_p$.
\end{abstract}

\maketitle
\setcounter{tocdepth}{1}
\tableofcontents

\section{Introduction}
\label{sec:intro}

In this paper we prove the following theorem.  Call a field $k$
\emph{semi-finite\/} if it is perfect and for any prime number
$\ell$, the maximal prime-to-$\ell$ extension of $k$ is
pseudo-algebraically closed (PAC) with Galois group $\Z_{\ell}$.  For example, finite fields and
pseudo-finite fields \cite{ax} are semi-finite.  Using the theory
of Weil restriction, it is possible to show that a finite
extension of a semi-finite field is semi-finite.

Let $k$ be a semi-finite field of characteristic exponent $p$ and
$K/k$ a field extension of transcendence degree $2$.  

\begin{thm}\label{T:main}
  Any $\alpha\in\Br(K)$ satisfies $\ind(\alpha) | \per(\alpha)^2$.
\end{thm}

\subsection*{A brief history}
We discuss the basic terminology and history
of the problem; a more extensive treatment can be found in
\cite{period-index-paper}.  A class $\alpha\in\Br(K)$ corresponds to
an isomorphism class of finite-dimensional central division algebras
$A$ over $K$.  We always have that $A\tensor\widebar
K\cong\M_n(\widebar K)$, so that $\dim_K A$ is a square.  The number
$n$ is called the \emph{index\/} of $A$ (written $\ind(A)$) and is a
crude measure of the complexity of $A$ as an algebra.  On the
other hand, as an element of the torsion group $\Br(K)$, $\alpha=[A]$ has an
order, called the \emph{period\/} of $A$ (written $\per(A)$).  This is
a measure of the complexity of $[A]$ as an element of the Brauer
group.

Using the cohomological interpretation of the Brauer group, one can
show that the period and index are related: the period always divides
the index and they have the same prime factors, so the index divides
some power of the period.  The period-index problem for the field $K$
is to determine the minimal value of $e$ such that $\ind(A) |
\per(A)^e$ for all finite-dimensional central division algebras $A$
over $K$.  This problem has proven extremely difficult, but a general
conjecture has emerged for certain fields $K$ (see page 12 of
\cite{colliot} or the Introduction of \cite{period-index-paper}): if
$K$ is a $C_d$-field then $e=d-1$.  This conjecture is known to hold when $K$ is 
\begin{itemize}
\item algebraically closed (Gauss)
\item of transcendence degree 1 over an algebraically closed field (Tsen)
\item of transcendence degree 2 over an algebraically closed field (de Jong \cite{MR2060023} with improvements in positive characteristic due to de Jong-Starr and Lieblich \cite{period-index-paper})
\item finite (Wedderburn)
\item of transcendence degree 1 over a finite field (Brauer-Hasse-Noether)
\item of transcendence degree 1 over a higher local field (Lieblich \cite{paiitbgoaas})
\end{itemize}
It is easy to see that the conjectural relation
is sharp in the
context of Theorem \ref{T:main}: if $k$ contains a primitive $n$th
root of unity with $n$ invertible in $k$ and $a\in k^\ast\setminus(k^\ast)^n$, then the bicyclic
algebra $(x,a)_n\tensor(x+1,y)_n$ is an element of $\Br(k(x,y))[n]$
whose index is strictly larger than its period.  (If $n$ is prime then
in fact this algebra has index $n^2$ -- i.e., it is a division
algebra.  For a discussion of this and numerous other examples, the reader is
referred to Section 1 of \cite{MR1923420}.)  

\subsection*{Our contribution}
In the present paper, we primarily address the case of surfaces over finite
fields (although our methods work over any semi-finite field, as defined above).  The potent combination of formal methods, geometry, and arithmetic available over a finite (resp.\ semi-finite) field make this a tractable class of $C_3$-fields.  
Theorem \ref{T:main} provides a first
example of a class of geometric $C_3$-fields for which the standard
period-index conjecture holds.
It is noteworthy that almost no
progress has been made for the other natural class of $C_3$-fields:
function fields of threefolds over algebraically closed fields.  In
fact, it is still unknown if there is any bound at all on the values
of $e$ that can occur in the relation $\ind(\alpha) | \per(\alpha)^e$
for $\alpha\in\Br(\C(x,y,z))$ (or any other fixed threefold). One might be tempted to fiber such a threefold as a surface over $\C(t)$ (or a similar field), but no attempt to do so to date has borne fruit.

\subsection*{Guide to this paper}
In Section \ref{sec:outline}, we give a broad outline of the strategy of the proof of Theorem \ref{T:main} and a guide to the contents of this paper.  Throughout this document we rely heavily on the theory of twisted sheaves.  Rather than develop the theory here from scratch, the reader is referred to \cite{period-index-paper} for background on these objects and their applications to the Brauer group.

\subsection*{A remark on parity}
The case of $\ell=2$ has special properties that necessitate significantly more complex arguments at several points. We have tried to relegate the bulk of the extra work needed in that case to clearly marked sections (Section \ref{sec:even-ram}, Section \ref{sec:numerical}, and the second half of Paragraph \ref{case2}). Readers interested in grasping the flow of the argument without getting bogged down in technical details are encouraged to omit those sections on a first reading.

\section{Outline of the proof of Theorem \ref{T:main}}
\label{sec:outline}

Let us briefly outline the strategy of the proof of the main theorem. In Section
\ref{sec:reductions}, we explain how to reduce to the case in which the class $\alpha$ 
has period $\ell$ prime to $p$, the ambient
field $K$ is the function field of a smooth projective geometrically connected
surface $X$ over a PAC base field $k$ containing a primitive
$\ell$th root of unity and having Galois group $\Z_\ell$, and the ramification divisor of $\alpha$ on $X$ has simple normal crossings. In Sections
\ref{sec:splitting} and \ref{sec:gerbe} we will explain how
to replace $X$ by a stack $\mc X$ over which $\alpha$ extends as a Brauer class
of period $\ell$ and then how to choose a good $\m_\ell$-gerbe $\ms X\to\mc X$
representing $\alpha$.  This puts us in a position to take the approach of
\cite{period-index-paper}: define a moduli space expressing the relation
$\ind(\alpha) | \ell^2$ and show that it has points.  To this end, in Section
\ref{sec:irred} we will prove the following crucial theorem.

\begin{thm}\label{T:part}
  There is an invertible sheaf $\ms N$ on $X$ and a geometrically integral open substack $\ms S$ of
  the (Artin) stack of coherent $\ms X$-twisted sheaves of rank
  $\ell^2$ and determinant $\ms N$.
\end{thm}

Since any such stack has affine stabilizers (the action of $\aut(\ms F)$ on
$\End(\ms F)$ being a faithful linear representation), it follows from
Proposition 3.5.9 of \cite{kresch-cycle} that there is a 
geometrically integral quasi-projective $k$-scheme $S$ and a smooth
morphism $\rho:S\to\ms S$.  This permits us to prove the main theorem.

\begin{proof}[Proof of Theorem \ref{T:main}]
  As $k$ is PAC and $S$ is geometrically
  irreducible (and thus a priori non-empty!), we know that $S(k)\neq\emptyset$.  Thus, there is an
  object of $\ms S$ over $k$, yielding an $\ms X$-twisted sheaf of
  rank $\ell^2$.  We conclude that $\ind(\alpha)$ divides $\ell^2$, as
  desired.
\end{proof}

The proof of Theorem \ref{T:part} is somewhat delicate and occupies
Sections \ref{sec:twisted-stacks} through \ref{sec:irred}.  It is
roughly modeled after O'Grady's proof of the irreducibility of the
space of stable sheaves on a smooth projective surface, but the
representation-theoretic content of the stack $\ms X$ makes the
problem more complicated.  In particular, there is a curve $\ms
D\subset\ms X$ of ``maximal stackiness'' that itself contains points of ``even more maximal stackiness'', and this stratification breaks up the moduli
problem into additional components.  Thus, after introducing the general moduli problems in Section \ref{sec:twisted-stacks}, we will produce objects over the base field in two steps.
\begin{enumerate}
\item In Section 8, we will produce a twisted
sheaf $\ms W$ of rank $\ell^2$ and suitable determinant supported on the curve $\ms D$. This in turn involves carefully tracking the properties of twisted sheaves supported over the singular points of $D$; in fact, this $0$-dimensional stratum is in some sense responsible for the $\ell^2$ in the period-index relation (rather than the $\ell$ proven in \cite{period-index-paper} for unramified classes).
\item In Sections \ref{sec:formal} through \ref{sec:irred} we will study
the moduli of $\ms X$-twisted sheaves whose restrictions to $\ms D$
are deformations of $\ms W$. Showing that the latter moduli space is
non-empty is a subtle lifting and formal gluing problem, carried out in Sections
\ref{sec:formal} through \ref{sec:nonempty}. The final O'Grady-type convergence result, showing that Galois orbits of components of moduli spaces eventually become singletons as a function of a discrete parameter, is proven in Section \ref{sec:irred}.
\end{enumerate}
It is tempting to hope that a deeper understanding of the link between the stratification by stabilizer and the inductive nature of the moduli problems might produce a strategy for working with higher dimensional varieties, but a precise formulation of such a principle is currently lacking.
\section{Reductions}
\label{sec:reductions}

In the following we fix $\alpha\in\Br(K)$. In this section we explain several
reductions that gradually make the problem increasingly geometric (and
tractable).

\begin{enumerate}
\item {\bf We may assume that $K$ is finitely generated over $k$ and that
    $k$ is algebraically closed in $K$}.  Indeed, a division algebra
  representing $\alpha$ is finitely generated over $K$, hence is
  defined over a finitely generated subfield of $K$ of transcendence
  degree $2$ over $k$.  The algebraic closure of $k$ in $K$ will be a
  finite extension and thus a semi-finite field.  Geometrically, we may
  assume that $K$ is the function field of a smooth projective 
  geometrically integral surface $X$ over $k$.
\item {\bf We may assume that $\alpha$ has prime period $\ell$}.  Indeed,
    suppose $\ell$ is a prime dividing $\per(\alpha)$.  (We do not yet assume
  that $\ell\neq p$.)  The class $(\per(\alpha)/\ell)\alpha$ has
  period $\ell$, hence by assumption we know that it has index
  dividing $\ell^2$.  There is thus a splitting field $K'/K$ of degree
  dividing $\ell^2$.  The class $\alpha_{K'}$ has period
  $\per(\alpha)/\ell$, whence by induction it has index dividing
  $(\per(\alpha)/\ell)^2$, so that there is a splitting field $K''/K'$
  of degree dividing $(\per(\alpha)/\ell)^2$.  We conclude that
  $K''/K$ is a splitting field of $\alpha$ of degree dividing
  $\per(\alpha)^2$, so that $\ind(\alpha) | \per(\alpha)^2$, as
  desired.
\item {\bf We may assume that the period $\ell$ of $\alpha$ is 
    distinct from $p$}.  Indeed, suppose $\alpha\in\Br(K)[p]$.  The
  absolute Frobenius $F:K\to K$ is a finite free morphism of degree
  $p^2$ and acts as multiplication by $p$ on $\Br(K)$.  It thus
  annihilates $\alpha$ by a field extension of degree $p^2$, as
  desired.  
\item {\bf We may assume that $k$ is PAC with Galois group $\Z_\ell$ and
    contains a primitive $\ell$th root of unity $\zeta$}.  Indeed, the
  algebra $k(\zeta)$ is contained in the maximal prime-to-$\ell$
  extension of $k$.  If $k'/k$ has degree $d$ relatively prime to
  $\ell$ and the restriction of $\alpha$ to $X\tensor k'$ has index
  dividing $\ell^2$ then $\alpha$ has index dividing $\ell^2d$, whence
  it has index dividing $\ell^2$ as its index is a power of $\ell$.
\item {\bf We may assume that the ramification divisor of $\alpha$ is a
    strict normal crossings (snc) divisor $D\subset X$}.  Indeed, if
  $D\subset X$ is the ramification divisor of $\alpha$ then we can
  find a blowup $b:\widetilde X\to X$ such that $\widetilde
  D:=b^{-1}(D)_{\text{\rm red}}$ is a snc divisor.  Since the
  ramification divisor of $\alpha$ on $\widetilde X$ is a subdivisor
  of $\widetilde D$, it must be snc.

\end{enumerate}
For the remainder of this paper, we will assume that $K=k(X)$ is the
function field of a smooth projective geometrically integral surface
over a PAC field $k$ with Galois group $\Z_\ell$ and that $\alpha$ is a
geometrically stable class of 
prime period $\ell$ different from
$p$ with snc ramification divisor $D=D_1+\cdots+D_n$.

\begin{notation}\label{N:prim} 
Because this condition will come up repeatedly, and we know of no existing term for it in the literature, we will call a field that is PAC with Galois group $\Z_\ell$, for an $\ell$ prime to the characteristic, an \emph{$\ell$-primitive\/} field.
\end{notation}

\section{Ramification}
\label{sec:splitting}

We briefly review the main aspects of the ramification theory of
Brauer classes as they apply in the present context.  A detailed
description of the theory is given in Section 3 of \cite{MR0321934} and
Section 2.5 of \cite{artin-dejong}.

\subsection{Splitting ramification with a stack}
The ramification theory of Brauer classes associates to each $D_i$ a
cyclic $\Z/\ell\Z$-extension $L_i/k(D_i)$, called the \emph{(primary)
  ramification\/}.  These extensions have the property that $L_i$ can
ramify only over points of $D_i$ which meet other components of $D$.
Moreover, if $q\in D_i\cap D_j$ then the ramification index of $L_i$
at $q$ equals the ramification index of $L_j$ at $q$.  This index is
called the \emph{secondary ramification\/}.

It is a basic consequence of the description of the ramification
extension that $\alpha$ restricted to $K(t^{1/\ell})$ is in the image
of $\Br(\ms O_{X,\eta_i}[t^{1/\ell}])$, where $t$ is a local equation
for $D_i$ and $\eta_i\in D_i$ is the generic point.  (This is just
Abhyankar's lemma for central simple algebras.  A proof can be found e.g.\ in Proposition 1.3 of \cite{MR1462850})  We can globalize this
splitting of the ramification if we use a stacky branched cover (that has the
advantage of not changing the function field) as follows.  

Let $r:\mc X\to X$ be the result of applying the $\ell$th root
construction (described, for example, in Section 2 of
\cite{cadman_using_2007}) to the components of the divisor $D$.  We
know that $\mc X$ is a smooth proper geometrically integral
Deligne-Mumford surface over $k$ and that $r$ is an isomorphism over
$X\setminus D$.  For each component $D_i$ of $D$, the reduced preimage
$\mc D_i\subset\mc X$ is a $\m_\ell$-gerbe over a stacky curve. The
reduced preimage $\mc D$ of $D$ in $\mc X$ is a {\it residual curve\/} of the
type studied in \cite{paiitbgoaas}. We will briefly review their properties in Section \ref{sec:residual-curves} below when discussing the existence of twisted sheaves over $\mc D$.

Since $\mc X$ is smooth, the restriction map $\Br(\mc X)\to\Br(K)$ is
injective.  It thus makes sense to ask if the element $\alpha$ belongs
the former group.  We recall the following fundamental result.

\begin{prop}
  The class $\alpha$ lies in $\Br(\mc X)[\ell]$.
\end{prop}

For a proof, the reader is referred to Proposition 3.2.1 of
\cite{paiitbgoaas}.

\subsection{Adjusting ramification when $\ell=2$}
\label{sec:even-ram}

In this section we discuss a method for reducing the algebraic
complexity of the ramification divisor for classes of period $2$.  A
similar type of phenomenon undoubtedly also holds for classes of odd
period, but it is significantly more complicated and will not help
with the main result.  The results described here are essentially
special cases of those in \cite{saltman_cyclic_2007}, with slight
changes for the present situation.

Fix a component $\mc D_i$ of the stacky locus in $\mc X$.  Recall that
a singular residual gerbe $\xi$ of $\mc D$ has the form
$(\B\m_2\times\B\m_2)_\kappa$ for a $2$-primary field $\kappa$ (see Notation \ref{N:prim}).  As
discussed in Section 4.3 of \cite{paiitbgoaas}, the class $\alpha_\xi$ is
uniquely determined by a pair of $\Z/2\Z$-cyclic \'etale
$\kappa$-algebras $L_1,L_2$ and an element $\gamma\in\m_2(\kappa)=\{1,-1\}$.  We will
always choose the identification so that the first factor is the
generic stabilizer of $\mc D_i$ (and the second is the generic
stabilizer of another component $\mc D_j$).  In particular, when
$\gamma=1$ we have that $L_1$ is
the specialization of the (\'etale) ramification extension of $D_i$.
Refining Saltman's terminology from \cite{saltman_cyclic_2007},
we say that
\begin{enumerate}
\item $\xi$ is \emph{cold\/} if $\gamma=-1$; otherwise, $\gamma=1$ and we say that 
\item $\xi$ is \emph{chilly\/} if $L_1$ and $L_2$ are both non-trivial
    $\Z/2\Z$-extensions;
\item $\xi$ is \emph{hot\/} if $L_1$ is non-trivial and $L_2$ is trivial;
\item $\xi$ is \emph{scalding\/} if $L_1$ is trivial and $L_2$ is non-trivial.
\end{enumerate}

We will call a singular point of the ramification divisor $D$ of
$\alpha$ cold, chilly, etc., if its reduced preimage in $\mc X$ is
cold, chilly, etc.  
The main result in this section is the following.

\begin{prop}\label{P:scalding}
  There is a proper smooth surface $X$ with function field $K$ such that the ramification
  divisor of $\alpha$ decomposes as $D=S+R$, where 
  \begin{enumerate}
  \item each component $S_i$ of $S$ contains only cold and scalding points, and (therefore) no two distinct components of $S$ intersect in a non-cold point;
  \item the components $R_i$ of $R$ are disjoint $(-2)$-curves, and each $R_i\cap D$ consists of precisely one $\Gamma(R_i,\ms O)$-rational hot point of $R_i$.
  \end{enumerate}
\end{prop}
\begin{proof}
  Choose any $X$ over which $\alpha$ has snc ramification divisor $D$
  and let $D_i\subset D$ be a component.  
  We will show that all non-cold points of $D_i\cap (D\setminus D_i)$ that are not scalding can be
  eliminated by blowing up.  

  \begin{lem}\label{L:chilly-blow}
    If $p\in D_i\cap D_j$ is a chilly point then the exceptional divisor of
    the blowup of $X$ at $p$ is not a ramification divisor.
  \end{lem}
  \begin{proof}
    Let $x$ and $y$ be local equations for $D_i$ and $D_j$ at $p$.  As
    explained in Proposition 1.2 of \cite{MR1462850}, we can write
    $\alpha_{K(\widehat{\ms O}_{X,p})}=\alpha'+(x,a)+(y,a)$, where $\alpha'\in\Br(\widehat{\ms O}_{X,p})$.
    A local equation for the blowup is given by $x=yX$, where $X$ is a
    coordinate on the exceptional divisor $E$.  We find that
    $\alpha_{\Bl_p\spec\widehat{\ms O}_{X,p}}=\alpha'+(X,a)$, which is
    unramified at $E$ (whose local equation is $y=0$).  
  \end{proof}

  \begin{lem}\label{L:hot-blow}
    If $p\in D_i\cap D_j$ is a hot point of $D_i$ then 
    \begin{enumerate}
    \item the exceptional divisor $E$ of $\Bl_p X$ is a ramification
      divisor with precisely one hot $\kappa(p)$-rational point;
    \item the intersection of the strict transform $\widetilde D_i$ with $E$
      is a chilly point of $\widetilde D_i$.
    \end{enumerate}
  \end{lem}
  \begin{proof}
    Arguing as in the proof of Lemma \ref{L:chilly-blow}, locally we
    have that $\alpha=\alpha'+(x,a)=\alpha'+(X,a)+(y,a)$.  Since $y$
    cuts out $E$, we see from the elementary ramification calculation
    of Section 3 of Chapter XII of \cite{MR554237} that $\alpha$ ramifies (with a constant ramification
    extension given by taking the square root of $a$) along $E$.  Moreover, $X$ locally cuts out $\widetilde
    D_i$, and we see that the point $X=y=0$ is chilly.  Finally,
    taking the other coordinate patch with $xY=y$, we see that
    $\alpha$ does not ramify along $(Y=0)$, which is $\widetilde D_j$,
    showing that $E\cap\widetilde D_j$ is hot, as claimed.
  \end{proof}
  Combining Lemmas \ref{L:chilly-blow} and \ref{L:hot-blow}, we see
  that we can blow up a chilly point to eliminate it and a hot point
  to create a pair consisting of a chilly point and a hot point on a
  $(-1)$-curve.  Blowing up again to eliminate the chilly point yields
  a $(-2)$-curve containing precisely one hot point rational over its
  constant field. At each step, a component of the original ramification divisor $D$ is made to contain fewer hot points, while $(-2)$-curves with single hot points are added. Once every component of $D$ has been ameliorated in this way, the ramification divisor assumes the form described in the statement of Proposition \ref{P:scalding}, completing the proof.

In order to see the ``therefore'' clause of the first part of the Proposition, note that a scalding point is hot on the transverse curve, so that if all hot points are eliminated from a set of components, only cold intersection points can remain among those components.
\end{proof}

\begin{cor}\label{C:ghost-dis}
If $D_i$ and $D_j$ are distinct components of $S$ that contain no cold points, then $D_i\cap D_j=\emptyset$.
\end{cor}
\begin{proof}
The divisors $D_i$ and $D_j$ only contain scalding and cold points. If they include no cold points, then all special points are scalding. But a scalding point (by definition!) is hot on the complementary component (as the ramification is non-trivial on the transverse curve). Thus, a point cannot simultaneously be scalding for two components, and thus $D_i\cap D_j=\emptyset$, as desired.
\end{proof}

\section{Fixing a uniformized $\m_\ell$-gerbe}
\label{sec:gerbe}

In this section we optimize the topological properties of a $\m_\ell$-gerbe representing the Brauer class $\alpha\in\Br(\mc X)$. The main result is Proposition \ref{P:vanishy}.

The Kummer sequence provides a short exact sequence 
$$\xymatrix{0\ar[r] &\Pic(\mc X)/\ell\Pic(\mc X)\ar[r]^-{c_1}& \H^2(\mc
X,\m_\ell)\ar[r]&\H^2(\mc X,\G_m)[\ell]\ar[r]& 0.}$$ We can thus choose a
lift of $\alpha$ to a class $\widetilde\alpha\in\H^2(\mc X,\m_\ell)$,
and we can modify this lift by classes coming from invertible sheaves
on $\mc X$ without changing the associated Brauer class.  We will
choose a particular lift which has a nice structure with respect to
the stacky locus $\mc D\subset\mc X$.  For each $i$, let
$\widebar\eta_i\to D_i$ be a geometric generic point and
$\widebar\xi_i\to\mc D_i$ be the reduced pullback to $\mc D_i$.  The
formation of the root construction provides a canonical isomorphism
$\widebar\xi_i\cong\B\m_{\ell,\widebar\eta_i}$.  

\begin{lem}\label{sec:fixing-m_ell-gerbe}
  Via the Kummer sequence, the invertible sheaf $\ms O_{\mc X}(\mc
  D_i)_{\widebar\xi_i}$ generates
  $\H^2(\widebar\xi_i,\m_\ell)=\Z/\ell\Z$.
\end{lem}
\begin{proof}
  It suffices to show that the action of $\m_\ell$ on the geometric
  fiber of $\ms O_{\mc X}(\mc D_i)$ is via a generator of the
  character group $\Z/\ell\Z$.
  Suppose $s\in\ms O_{X,\eta_{D_i}}$ is a local uniformizer for $D_i$.  In
  local coordinates at the generic point of $\mc D_i$ we can realize
  $\mc X$ as the quotient stack $[\spec(\ms
  O_{X,\eta_{D_i}}[t]/(t^\ell-s))/\m_\ell]$ with $\m_\ell$ acting on
  $t$ by scalar multiplication.  But $t$ is a local generator of
  $\ms O_{\mc X}(-\mc D_i)$, so the action of $\m_\ell$ on the fiber
  is via the inverse of the natural character, and this generates the
  character group.
\end{proof}

\begin{prop}\label{P:vanishy}
  There is a lift $\widetilde\alpha\in\H^2(\mc X,\m_\ell)$ such that
  for all $i$ the restriction
  $\widetilde\alpha|_{\widebar\xi_i}$ vanishes in $\H^2(\widebar\xi_i,\m_\ell)$.
\end{prop}
\begin{proof}
  Choose any lift $\widetilde\alpha'$.  By Lemma
  \ref{sec:fixing-m_ell-gerbe}, for each $i$ there exists $j_i$ such
  that the restriction of $\widetilde\alpha'$ to $\widebar\xi_i$ has
  the same class as $\ms O_{\mc X}(j_i\mc D_i)$.  Setting
  $\widetilde\alpha=\widetilde\alpha'-c_1(\ms O_{\mc X}(-\sum_i j_i\mc
  D_i)$ gives the desired result.
\end{proof}

\begin{notation}
  For the rest of this paper we fix a $\m_\ell$-gerbe $\pi:\ms X\to\mc
  X$ whose associated cohomology class $[\ms X]$ maps to
  $\alpha\in\Br(K)$ and has the property that for each $i=1,\ldots,n$,
  the pullback $\ms X\times_{\mc X}\widebar\xi_i$ is isomorphic to
  $\widebar\xi_i\times\B\m_\ell$.  We will write $\ms D_i$ for the
  reduced preimage of $D_i$ in $\ms X$ and $\ms D$ for the reduced
  preimage of $D$.  There is an equality $\ms D=\sum\ms D_i$ of
  effective (snc) Cartier divisors.
\end{notation}

We will also need to define a second Chern class and
Castelnuovo-Mumford regularity for $\ms X$-twisted sheaves.  One way
to do this is via a projective uniformization of $\ms X$.  Let
$u:Z\to\ms X$ be a finite flat cover by a smooth projective surface.
(That such a uniformization exists follows from Theorem 1 and Theorem
2 of \cite{MR2026412}, combined with Gabber's Theorem that $\Br$ and
$\Br'$ coincide for quasi-projective schemes, a proof of which may be
found in \cite{dejong-gabber}.)

\begin{defn}
  Given a coherent $\ms X$-twisted sheaf $\ms F$, the \emph{second
    $u$-Chern class of $\ms F$\/} is $c(\ms F):=\deg c_2(u^\ast\ms
  F)$.  The \emph{$u$-Castelnuovo-Mumford regularity of $\ms F$\/}, written
  $r(\ms F)$, is the Castelnuovo-Mumford regularity of $u^\ast\ms F$.
\end{defn}

\section{Residual curves, ghost components, and intersection numbers}
\label{sec:numerics}
In this section, we briefly review the theory of residual curves introduced in \cite{paiitbgoaas} and use it to refine our understanding of the geometry of $\ms X\to\mc X$ over the curve $\mc D$.

\subsection{Residual curves and ghost components}
\label{sec:residual-curves}

The curve $\mc D$ introduced in Section \ref{sec:splitting} has a very special form: it is a tame Deligne-Mumford stack of dimension $1$ over a field $\kappa$, whose coarse moduli space $D=\cup D_i$ is an snc curve, and there are Zariski $\m_\ell$-gerbes $\mc D_i\to D_i$ such that $$\mc D\cong \mc D_1\times_D\cdots\times_D\mc D_n.$$ In particular, each component $\mc D_i$ is a Zariski $\m_\ell$-gerbe over a smooth Deligne-Mumford curve that has a divisor $\mc S_i$ supporting the entire locus with non-trivial automorphisms, and the reduced structure on $\mc S_i$ makes it isomorphic to $\B\m_\ell\times S_i$ for some finite \'etale $\kappa$-scheme $S_i$. In particular, the residual gerbes of $\mc D_i$ are isomorphic to $\B\m_{\ell, L}$ or $\B\m_{\ell, L}\times\B\m_{\ell, L}$ for $L$ a finite separable extension of $\kappa$. Curves like $\mc D$ are called \emph{residual curves\/} in \cite{paiitbgoaas}, and they are precisely the curves that split the residues of Brauer classes on (suitable birational models of) surfaces.

\begin{notation}
We will write $\kappa_i$ for $\Gamma(\mc D_i,\ms O)$; each $\kappa_i$ is $\ell$-primary (Notation \ref{N:prim}).
\end{notation}

As explained in Section \ref{sec:gerbe}, we also have a $\m_\ell$-gerbe $\ms D_i\to\mc D_i$ parametrizing the restriction of the extension of our Brauer class $\alpha$. This class gives rise to Brauer classes over the residual gerbes. 
\begin{notation}
The calculation of the Brauer group of $\B\m_\ell$ (see Section 4 of \cite{paiitbgoaas}) associates to each $\ms D_i\to \mc D_i$ a cyclic extension that we will always write as $R_i\to D_i$. (This is in fact the same as the classical ramification extension when thinking of $D$ as the ramification divisor of $\alpha$.)
\end{notation}

\begin{defn}
A component $D_i$ is a \emph{ghost component\/} if the extension $R_i\to D_i$ induced by a cyclic extension of $\kappa_i$.
\end{defn}
Equivalently, a ghost component is one whose ramification extension is geometrically trivial.

There is a sheaf-theoretic characterization of ghost components as follows. Call a sheaf $G$ on $\mc D_i$ \emph{isotypic of type $c$\/} if its restriction to the generic gerbe $G|_{\B\m_{\ell,\kappa(D_i)}}$ is isomorphic to the sheaf on $\B\m_{\ell,\kappa(D_i)}$ associated to a representation of the form $(\chi^{\tensor c})^{\oplus N}$, where $\chi:\m_\ell\to\G_m$ is the natural inclusion character.

\begin{lem}\label{L:ghost-sheaf}
A component $D_i$ is a ghost component if and only if for every locally free $\ms D_i$-twisted sheaf $F$ and every algebraically closed extension $\kappa_i\subset K$, the sheaf $F\tensor_{\kappa_i}K$ admits a direct sum decomposition
$$F\cong 	F_0\oplus\cdots\oplus F_{\ell-1}$$
such that for every pair $a, b\in\{0,\ldots,\ell-1\}$, the sheaf $\pi_\ast\shom(F_a, F_b)$ is isotypic of type $b-a$. Moreover, such a decomposition is unique up to reordering the summands and applying summand-wise isomorphisms.
\end{lem}
\begin{proof}
This is proven in Lemma 5.5 and Proposition 5.1.8 of \cite{paiitbgoaas}.
\end{proof}

\begin{notation}\label{N:eigendecomp}
A decomposition as in Lemma \ref{L:ghost-sheaf} is called an \emph{eigendecomposition\/}.
\end{notation}

\subsection{Numerical consequences of Section \ref{sec:even-ram} when $\ell=2$}
\label{sec:numerical}
In this section we explain some intersection-theoretic consequences of the ramification configuration created in Section \ref{sec:even-ram} in the case of $\ell=2$.
\begin{remark}
This is the only example we know of at the moment where (1) the knowledge that a divisor is the ramification divisor of a Brauer class, and (2) the knowledge of which components disappear from the ramification of the class over the algebraic closure of the base field, together yield intersection-theoretic consequences for the underlying divisor.
\end{remark}

\begin{prop}\label{P:parity}
  If $D_i$ is a ghost component then both $D_i^2$ and $D_i\cdot R$ are even.
\end{prop}
The parity is computed by viewing the intersection as a scheme over the field $\kappa_i$, the field of constants of $D_i$, not over the original base field. That is to say, if $D_i$ is defined over a quadratic extension, this does not mean that every intersection number is even by default.
\begin{proof}
  The proof breaks into two subcases:
  $D_i\subset S$ and $D_i\subset R$.  In the latter case we already
  know that $D_i$ is a $(-2)$-curve, and that $D_i^2=D_i\cdot R$.  Thus, we will assume for the
  rest of this proof that $D_i\subset S$.

  Let $\{r_1,\ldots,r_a\}=D_i\cap(\cup_{j\neq i}D_j)$.  By the
  reduction of Section \ref{sec:even-ram}, each $r_j$ is a scalding
  point, so that the restriction of $L$ to $r_j$ is trivial.  Since
  $L$ is a pullback from $k_i$, we see that each residue field
  $\kappa(r_j)$ has even degree over $k_i$.  In particular, we immediately see that $D_i\cdot R$ is even.  

It remains to show that $D_i^2$ is even.  Write $\mc D_i=[\ms
  O_{D_i}(D_i)]^{1/2}\times_{D_i}\mc C_i$, where $\mc C_i\to D_i$ is
  the root construction applied to $D_i\cap R$ and $[\ms
  O_{D_i}(D_i)]^{1/2}$ is the stack of square-roots of $\ms
  O_{D_i}(D_i)$ (i.e., the gerbe representing the image of $\ms
  O_{D_i}(D_i)$ under the Kummer boundary map $\H^1(D_i,\G_m)\to\H^2(D_i,\m_2)$).  By class field theory
  and the fact that each $r_j$ has even degree over $k_i$, we know
  that there is a Brauer class $\beta\in\Br(\mc C_i)$ whose
  ramification extension over each $r_j$ is non-trivial.  
  \begin{lem}
    With the immediately preceding notation, there is a class
    $\gamma\in\Br([\ms O_{D_i}(D_i)]^{1/2})[2]$ such that
    $\alpha-\beta_{\mc D_i}=\gamma_{\mc D_i}$.
  \end{lem}
  \begin{proof}
    The Leray spectral sequence for the projection morphism $\mc
    D_i\to[\ms O_{D_i}(D_i)^{1/2}]$ yields an exact sequence 
$$0\to\Br([\ms O_{D_i}(D_i)]^{1/2})\to\Br(\mc
D_i)\to\bigoplus_j(\kappa(r_j)^\times\tensor\Z/2\Z)\oplus\Z/2\Z$$
in which the rightmost map is the sum of the projections to the second
two factors in the natural decompositions
$\Br(\xi_j)\simto\kappa(r_j)^\times\tensor\Z/2\Z\oplus\kappa(r_j)^\times\tensor\Z/2\Z\oplus\Z/2\Z$.
By assumption, for each $j$ the third projection of $\alpha$ (the secondary
ramification) is trivial, while the second projection (``primary
ramification along the other branch'') is the same for $\alpha$ and
$\beta$.  Thus, the difference $\alpha-\beta_{\mc D_i}$ lies in the
image of $\Br([\ms O_{D_i}(D_i)]^{1/2})$, as desired.
  \end{proof}
Since $\alpha$ is
  ramified along $D_i$, the class $\gamma$ must be non-zero.  This
  gives us numerical information about $D_i^2$, as the following lemma
  shows.

\begin{lem}\label{L:parity}
  Suppose $f:\ms C\to C$ is a $\m_2$-gerbe on a proper smooth curve over
  a finite or PAC field $\kappa$.  If $$\ker(\Br(\ms C)\to\Br(\ms
  C\tensor\widebar\kappa))\neq 0$$ then the image of $[\ms C]$ under the
  degree map $\H^2(C,\m_2)\to\Z/2\Z$ is $0$.  In particular, over
  $\widebar\kappa$ there is a invertible $\ms C$-twisted sheaf $N$
  such that $N^{\tensor 2}\cong\ms O$.  Finally, any invertible $\ms C$-twisted
  sheaf $N$ has the property that $$f_\ast(N^{\tensor 2})\in\Pic(C)$$ has even degree.
\end{lem}
Note that the pullback map $f^\ast:\Pic(C)\to\Pic(\ms C)$ is injective, so that the last statement really means that $N^{\tensor 2}$ can be canonically identified with an invertible sheaf on $C$, and this sheaf has even degree (over the field $\kappa$).
\begin{proof}
  The Leray spectral sequence shows that the kernel in question is
  isomorphic to the kernel of the edge map $\H^1(\spec\kappa,\Pic_{\ms
    C/\kappa})\to\H^3(\spec\kappa,\G_m)$.  Thus, we
  certainly must have that $\H^1(\spec\kappa,\Pic_{\ms C/\kappa})\neq
  0$.  The degree map defines an exact sequence
$$0\to\Pic^0_{\ms C/\kappa}\to\Pic_{\ms C/\kappa}\to\Z\to 0,$$
from which we deduce that $\H^1(\spec\kappa,\Pic^0_{\ms C/\kappa})\neq
0$.  By Lang's Theorem (resp.\ the PAC property), this is only possible if the group scheme $\Pic^0_{\ms
  C/\kappa}$ is disconnected, which implies that there is an
invertible $\ms C\tensor\widebar\kappa$-twisted sheaf of degree $0$. This gives the first statement of the Lemma by Proposition 3.1.2.1(iv) of \cite{period-index-paper}.  
Since any degree $0$ invertible sheaf on $C\tensor\widebar\kappa$ is a
square, the second statement of the Lemma follows.  The final
statement follows from the fact that any two invertible $\ms
C$-twisted differ by an invertible sheaf on $C$, so that their squares
differ by a square.  Since there is one whose square has degree $0$
(over $\widebar k$), we conclude that they all have squares of even degree.
\end{proof}

Consider the sheaf $\ms O_{\mc D_i}(\mc D_i)$.  This is an invertible
$\mc D_i$-twisted sheaf, and we conclude from Lemma \ref{L:parity}
that its square has even degree.  But its square is isomorphic to $\ms
O_{D_i}(D_i)|_{\mc D_i}$, so it has degree $D_i^2$, completing the
proof of Proposition \ref{P:parity}.
\end{proof}

\begin{cor}\label{C:ghost-det}
For each ghost component $D_i\subset S$ we have that $D_i\cdot S$ is even.
\end{cor}
\begin{proof}
By Corollary \ref{C:ghost-dis}, no two ghost components of $S$ intersect. The result then follows from Proposition \ref{P:parity}.
\end{proof}

\section{Stacks of twisted sheaves}
\label{sec:twisted-stacks}

The purpose of this section is to introduce the general moduli problems that we will study in the sequel. Given a
closed substack $Y\to\mc X$, let $\ms Y\to Y$ denote the pullback
$Y\times_{\mc X}\ms X$.  The uniformization $u$ of the previous
section induces a uniformization $Z\times_{\ms X}\ms Y\to\ms Y$ by a
projective scheme.  Fix an invertible sheaf $\ms N$ on $\ms Y$.

\begin{defn}
  The stack of $\ms Y$-twisted sheaves of rank $r$ and determinant $\ms N$
  will be denoted $\ms M_{\ms Y}(r,\ms N)$.
\end{defn}

The representation theory of the stabilizers of $\mc X$ puts natural
conditions on the sheaf theory of $\ms X$.  We distinguish a weak
condition that will be important in what follows.  First, recall that
the root construction canonically identifies a singular residual gerbe of $\mc
D$ with residue field $L$ with $\B\m_{\ell,L}\times\B\m_{\ell,L}$.
The two resulting maps
$\B\m_{\ell,L}\to\B\m_{\ell,L}\times\B\m_{\ell,L}$ arising from the
inclusion of the factor groups will be called the \emph{distinguished maps\/}.
Given an
algebraically closed field $\kappa$, we will say that a representable
morphism $x:\B\m_{\ell,\kappa}\to\mc X$ is a \emph{distinguished
  gerbe\/} if the image of $x$ lies in the smooth locus of
$\ms D$ or if $x$ factors through a distinguished map to a singular
residual gerbe of $\mc D$.

Given a distinguished gerbe $x:\B\m_{\ell,\kappa}\to\mc D$,
the pullback $\ms X_x$ has trivial cohomology class, so that there is
an invertible $\ms X_x$-twisted sheaf $\ms L$.

Now let $\mc S$ be an inverse limit of open substacks of a fixed closed substack of $\mc X$.
Write $\ms S=\ms X\times_{\mc X}\mc S$.  (The relevant examples: open
subsets of $\mc X$, open subsets of $\mc D$, and generic points of
components of $\mc D$.)  A distinguished gerbe of $\mc S$ is
a distinguished gerbe of $\mc X$ factoring through all open
substacks in the system
defining $\mc S$.

\begin{defn}\label{D:reg}
  A coherent $\ms S$-twisted sheaf $\ms V$ is 
  \begin{itemize} 
  \item \emph{regular\/} if for all
  distinguished gerbes $x$ of $\mc S$, the sheaf $\ms V_{\ms
    X_x}\tensor\ms L^\vee$ is the coherent sheaf associated to a
  direct sum $\rho^{\oplus m}$ for some $m$, where $\rho$ is the
  regular representation of $\m_\ell$;
  \item \emph{biregular\/} if $\ms V$ is regular and if for all geometric residual gerbes $\widebar{\xi}\to\mc S$ and all
invertible sheaves $L$ on $\widebar{\xi}$, we have that $\ms V$ and $\ms
V\tensor L$ are isomorphic;
\item \emph{happily biregular\/} if it is biregular and for all residual gerbes $\xi\to\ms S$, the restriction $\ms V|_\xi$ has the property that the group-scheme $\Aut_0(\ms V|_\xi)$ parametrizing determinant-preserving automorphisms is geometrically connected over $\Gamma(\xi,\ms O)$.
\end{itemize}
\end{defn}

\begin{remark}
The term ``geometric residual gerbe'' in the definition of biregularity means that $\widebar{\xi}$ factors through an isomorphism with the basechange of a residual gerbe to an algebraically closed field containing its field of definition.
\end{remark}
\begin{remark}
If $\ms S$ is a $\m_\ell$-gerbe over $\B\m_\ell\times\B\m_\ell$
admitting an invertible twisted sheaf $\Lambda$, it is easy to check
that any biregular $\ms S$-twisted sheaf of rank $\ell^2$ is isomorphic to $\Lambda$
tensored with the regular representation of $\m_\ell\times\m_\ell$.
As a consequence, if $\ms S$ is a $\m_\ell$-gerbe over
$\B\m_\ell\times\B\m_\ell$ with geometrically trivial Brauer class
then there is exactly one isomorphism class of biregular $\ms
S$-twisted sheaves of rank $\ell^2$.
\end{remark}
\begin{remark}
  A locally free $\ms D$- or $\ms X$-twisted sheaf is (bi)regular if and
  only if its restriction to the singular residual gerbes of $\ms D$
  is (bi)regular.
\end{remark}

The following result on regular sheaves will be important in Section
\ref{sec:formal}.

\begin{lem}\label{L:regzar}
  Given $\ms S$ as in the paragraph preceding Definition \ref{D:reg} which
  is contained in $\ms X\setminus\Sing(\ms D)$, any two regular
  locally free $\ms S$-twisted sheaves $\ms V_1$ and $\ms V_2$ of the
  same rank $r$ are Zariski-locally isomorphic.
\end{lem}
\begin{proof}
  It suffices to prove the result when $\mc S$ is the preimage in $\mc
  X$ of the spectrum of a local ring $A$ of $X$ at a point disjoint
  from the singular locus of $D$.  Let $p\in\spec A$ be the closed
  point; the reduced fiber $\xi$ of $\ms S$ over $p$ is either isomorphic to
  $p$ or to $\B\m_{\ell,\kappa}$, where $\kappa$ is the residue field
  of $p$.

  Since $A$ is affine and $\ms X$ is tame, the restriction map
  $\hom(\ms V_1,\ms V_2)\to\hom(\ms V_1|_{\xi},\ms V_2|_{\xi})$ is
  surjective.  Moreover, by Nakayama's lemma we have that a map $\ms
  V_1\to\ms V_2$ is an isomorphism if and only if its restriction to
  $\xi$ is an isomorphism.  Thus, we are reduced to proving the result
  when $\ms S=\xi$, which we assume for the rest of this proof.  

  The regularity condition shows that the open subset $\isom(\ms
  V_1,\ms V_2)$ of the affine space $\hom(\ms V_1,\ms V_2)$ has a
  point over the algebraic closure of $\kappa$.  Since $\kappa$ is
  infinite (by the reductions in Section \ref{sec:reductions}) the result follows, as nonempty open subsets of affine
  spaces over infinite fields always have rational points.  (As an amusing aside: if
  $\kappa$ is finite, then the nonemptiness of the locus shows that $\isom(\ms
  V_1,\ms V_2)$ is a torsor under the algebraic group $\aut(\ms V_1)$.
  But this group is an open subset of an affine space and therefore 
  connected.  Lang's theorem implies that any torsor is trivial and
  thus there is an isomorphism defined over the base field $\kappa$ in this case as well.)
\end{proof}

It is a standard computation in $K$-theory that regularity is an open
condition in the stack of locally free $\ms X$-twisted sheaves.  We
will study certain stacks of regular $\ms X$-twisted sheaves in order
to prove Theorem \ref{T:part}.

\begin{defn}
  Given a sheaf $\ms F$ with determinant $\ms N$, an
  \emph{equideterminantal\/} deformation of $\ms F$ is a family $\mf
  F$ over $T$ with a fiber identified with $\ms F$ and a global
  isomorphism $\det\mf F\simto\ms N_T$ reducing to the given
  isomorphism $\det\ms F\simto\ms N$ on the fiber.
\end{defn}

Given an $\ms X$-twisted sheaf $\ms F$, let $\ext^i_0(\ms F,\ms F)$
denote the kernel of the trace map $\ext^i(\ms F,\ms F)\to\H^i(\ms
X,\ms O)$.  When $\ms F$ has rank relatively prime to $p$, the formation
of traceless $\ext$ is compatible with Serre duality, so that
$\ext^i_0(\ms F,\ms F)$ is dual to $\ext^{2-i}_0(\ms F,\ms
F\tensor\omega_{\ms X})$.  In particular, $\ext^2_0(\ms F,\ms F)$ is
dual to the space of traceless homomorphisms $\hom_0(\ms F,\ms
F\tensor\omega_{\ms X})$.

\begin{defn}
  An $\ms X$-twisted sheaf $\ms V$ is \emph{unobstructed\/} if $\ext^2_0(\ms
  V,\ms V)=0$.
\end{defn}

\begin{lem}\label{L:U}
  Given an invertible sheaf $\ms N$, the set of unobstructed
  torsion-free coherent $\ms X$-twisted sheaves of rank $\ell^2$ and
  determinant $\ms N$ is a smooth open substack $\ms U$ of the stack
  of all $\ms X$-twisted coherent sheaves of determinant $\ms N$.
\end{lem}
\begin{proof}
  Since $\ell^2$ is invertible in $k$, given a $k$-scheme $T$ and a
  $T$-flat family of coherent $\ms X$-twisted sheaves $\ms V$ on $\ms
  X\times T$, the trace map $\R(\pr_2)_\ast\rshom(\ms V,\ms
  V)\to\R(\pr_2)_\ast\ms O_{\ms X\times T}$ splits, so that there is a
  perfect complex $\ms K$ on $T$ with $\R(\pr_2)_\ast\rshom(\ms V,\ms
  V)\cong\R(\pr_2)_\ast\ms O_{\ms X\times T}\oplus\ms K$.  A fiber
  $\ms V_t$ is unobstructed if and only if the derived base change
  $\ms K_t$ has trivial second cohomology.  By cohomology and base
  change, there is an open subscheme $U\subset T$ such that for all
  $T$-schemes $s:T'\to T$, we have that $\ms H^2(\mathbf L
  s^\ast\R(\pr_2)_\ast\ms K)=0$ if and only if $s$ factors through
  $U$.  These $U$ define the open substack $\ms U$ of unobstructed
  twisted sheaves.

  The smoothness of $\ms U$ is a consequence of the fact that the
  association $\ms V\rightsquigarrow\ext^2_0(\ms V,\ms V)$ is an obstruction
  theory in the sense of \cite{MR0260746} for the moduli problem of
  equideterminantal deformations, and the fact that trivial obstruction
  theories yield smooth deformation spaces.
\end{proof}

\begin{lem}\label{L:U-nonempty}
  The stack $\ms U$ of Lemma \ref{L:U} contains a point $[\ms V]$ such
  that the quotient $\ms V^{\vee\vee}/\ms V$ is the pushforward of an
  invertible twisted sheaf supported on a finite reduced closed
  substack of $\ms X\setminus\ms D$.  In
  particular, $\ms U$ is nonempty.
\end{lem}
\begin{proof}
  This works just as in the classical case.  Let $x\in X\setminus D$ be a general
  closed point.  Serre duality shows that $\ext^2_0(\ms V,\ms V)$ is
  dual to the space $\hom_0(\ms V,\ms V\tensor K_X)$ of traceless
  homomorphisms.  Taking a general length $1$ quotient $\ms V_x\to Q$
  yields a subsheaf $\ms W\subset\ms V$ such that $\hom_0(\ms W,\ms
  W\tensor K_X)\subsetneq\hom_0(\ms V,\ms V\tensor K_X)$.  The reader
  is referred to the first paragraph of the proof of Lemma \ref{L:sliding-cok}
  below for more details.
\end{proof}

\begin{lem}\label{L:cover-bdd}
  The open substack $\ms U(c,N)$ parametrizing unobstructed $\ms
  X$-twisted sheaves $\ms F$ of trivial determinant such that $\deg
  c_2(u^\ast\ms F)=c$ and $r(u^\ast\ms F)\leq N$ is of finite type
  over $k$.
\end{lem}
\begin{proof}
  By the methods of Section 3.2 of \cite{orbi}, this is reduced to the
  same statement on $Z$, where this follows from Th\'eor\`eme
  XIII.1.13 of \cite{MR0354655}.
\end{proof}

\section{Existence of $\ms D$-twisted sheaves}
\label{sec:existence-d}

In this section we start the bootstrapping process that will yield the proof of Theorem \ref{T:part} by proving that there are suitable twisted sheaves supported on $\ms D$. Much of the theory in this section is an outgrowth of the theory developed in Sections 4 and 5 of \cite{paiitbgoaas}. However, the results there are inadequate for our purposes when $\ell=2$, so we have recast some of them in a more flexible way here, in addition to proving the additional results needed for the even case. In an attempt to balance exposition with efficiency, we have tried to make the underlying ideas clear while referring to specific proofs in \cite{paiitbgoaas} when they can be dropped in here verbatim (or almost verbatim).

\subsection{Statement of the main result}
The goal of this section is the following.

\begin{thm}\label{T:sheaf-on-D}
There is an invertible sheaf $\ms N$ on $\mc X$ and a biregular $\ms D$-twisted sheaf of rank $\ell^2$ and determinant $\ms N|_{\mc D}$.
\end{thm}
The choice of $\ms N$ will depend upon the parity of $\ell$. This choice could be made uniform, but there are a few subtle cohomological implications of existence results with different determinants. We will not discuss those here, but we wish the record to reflect the more flexible version of the results for potential future users.

Before attacking Theorem \ref{T:sheaf-on-D}, we review and update some of the material of Sections 4 and 5 of \cite{paiitbgoaas}. As we will see, a single method works for all values of $\ell$, but the case of $\ell=2$ introduces one essential complexity related to the determinant.

\subsection{Biregular twisted sheaves over singular residual gerbes}
\label{sec:biregular}
We recall a fundamental result proved in Section 4 of \cite{paiitbgoaas}, recasting results of Saltman described in  \cite{saltman_cyclic_2007}. Write $\xi=\B\m_\ell\times\B\m_\ell$, over an $\ell$-primitive field $L$ (see Notation \ref{N:prim}).

\begin{remark}[Remark on hypotheses]
Section 4 of \cite{paiitbgoaas} assumes that the field in question is finite, while we work with $\ell$-primitive fields. The arguments carry over unchanged; the key properties that make the proofs work in Section 4 of \cite{paiitbgoaas} are the fact that $\Br(L)=0$ and that $\H^1(\spec L, \m_\ell)=\Z/\ell\Z$, both of which are ensured by the $\ell$-primitive hypothesis. 
\end{remark}

\begin{prop}[Modified Proposition 4.3.7 of \cite{paiitbgoaas}]\label{P:bireg}
There is a canonical isomorphism of groups
$$\Br(\xi)\simto (L^\times/(L^\times)^\ell)^2\times\m_\ell(L)\cong\Z/\ell\Z\times\Z/\ell\Z\times\m_\ell(L).$$
Moreover, given a $\m_\ell$-gerbe $\ms X\to\xi$ parametrizing a Brauer class $(A, B, \gamma)$ under this isomorphism, 
\begin{enumerate}
\item any two biregular $\ms X$-twisted sheaves of rank $\ell^2$ are isomorphic;
\item the biregular $\ms X$-twisted sheaf of rank $\ell^2$ is happily biregular (Definition \ref{D:reg}) if and only if $\gamma=1$.
\end{enumerate}
\end{prop}
\begin{proof}
As described in Section 4.3 of \cite{paiitbgoaas}, biregular sheaves with Brauer class $(A,B,\gamma)$ correspond to pairs of operators $\alpha$ and $\beta$ on a vector space $V$ such that 
\begin{enumerate}
\item $\alpha^\ell = A$, $\beta^\ell = B$, and $\alpha\beta=\gamma\beta\alpha$ as endomorphisms of $V$;
\item over $\widebar{L}$ with chosen elements $A^{1/\ell}$ and $B^{1/\ell}$, the operators $\frac{1}{A^{1/\ell}}\alpha$ and $\frac{1}{B^{1/\ell}}\beta$, viewed as actions of $\m_\ell$, give multiples of the regular representation of $\m_\ell$.
\end{enumerate}
When $\gamma\neq 1$ (the case called ``cold'' by Saltman in \cite{saltman_cyclic_2007}), an isomorphism between the cyclic algebra $(A,B)_\gamma$ and $\M_\ell(L)$ (which exists because, by assumption, $L$ has trivial Brauer group) gives a biregular twisted sheaf $V$ of rank $\ell$. Moreover, the endomorphism ring of this sheaf is identified with the center of $(A,B)_\gamma$, which is simply scalars -- that is, this sheaf is geometrically simple. The Skolem-Noether theorem implies that this is in fact the only biregular twisted sheaf of rank $\ell$, up to isomorphism. Since $\ell$ is invertible in $L$, the category of coherent sheaves is semisimple, and thus $V^{\oplus\ell}$ is the only biregular twisted sheaf of rank $\ell^2$. 

The automorphism group scheme is identified with $\GL_\ell$, but action of a matrix $M$ on the determinant of $V$ is by the $\ell$th power of the determinant of $M$. That is, the group scheme parametrizing determinant-preserving automorphisms is never geometrically connected for cold gerbes.

On the other hand, if $\gamma=1$, the algebra $L[x,y]/(x^\ell-A, y^\ell-B)$ admits an action as described, giving a biregular twisted sheaf $W$ of rank $\ell^2$. Extending scalars to $\widebar L$, the operators $x/A^{1/\ell}$ and $y/B^{1/\ell}$ make $W$ isomorphic to the regular representation of $\m_\ell\times\m_\ell$. This replacement is equivalent to choosing a trivialization of the Brauer class over $\widebar L$ (i.e., choosing $\ell$th roots for $A$ and $B$). The automorphisms of the regular representation that preserve the determinant are isomorphic to $\G_m^{\ell^2-1}$, which is geometrically connected, making $W$ happily biregular, as desired. Any biregular twisted sheaf in this situation is isomorphic to $W$: they are isomorphic over $\widebar L$, and the isomorphisms are a torsor under a geometrically connected group scheme, which must have a point over $L$ (since $L$ is PAC by the $\ell$-primitive assumption).
\end{proof}

\subsection{Uniform twisted sheaves and their moduli}
\label{sec:uniform-sheaves}
Much of this section is a streamlined and updated form of the relevant material in Section 5 of \cite{paiitbgoaas}

Fix a regular locally free $\ms D$-twisted sheaf $\ms V$ of rank $\ell^2$. Write $\mc V_i$ for $\ms V|_{\ms D_i}$. 

\begin{defn}
The sheaf $\ms V$ is \emph{uniform\/} if for each ghost component $D_i$, the sheaf $\ms V|_{\ms D_i\tensor_{\kappa_i}\widebar\kappa_i}$ admits an eigendecomposition (see Notation \ref{N:eigendecomp}) $F_0\oplus \cdots F_{\ell-1}$ in which each component has rank $\ell$, and all sheaves $\pi_\ast\shom(F_a,F_b)$ have degree $0$.
\end{defn}
That is, each component of the eigendecomposition has ``the same degree'' (without having to quibble about the definition of degrees of sheaves on gerbes). The uniform condition is a natural one to impose, since if one wants to find a twisted sheaf over the base field $\kappa$ of $D$, it must be Galois invariant on each component. Since the ramification extensions are not trivial to begin with, one can see that the Galois group must cyclically permute the components of an eigendecomposition, forcing equality of degrees.

\begin{remark}
In Section 5 of \cite{paiitbgoaas}, the uniform condition included triviality of the determinant. As we will demonstrate below, when $\ell=2$, it is essential that one allow other determinants. In fact, it was precisely this trivial determinant condition in \cite{paiitbgoaas} that forced $\ell$ to be odd and led to various gymnastics. We avoid such unnecessary exercise here.
\end{remark}

Uniform sheaves form an open substack $$\ms M_{\ms D}^\unif(\ell^2, \ms N)\subset\ms M_{\ms D}(\ell^2, \ms N),$$ in the notation of the beginning of Section \ref{sec:twisted-stacks}. The main result on uniform twisted sheaves is the following.

\begin{prop}\label{P:uniform-irred}
For a fixed invertible sheaf $\ms N$ on $\mc D$, the stack $\ms M_{\ms D}^\unif(\ell^2,\ms N)$ is geometrically integral if it is non-empty.
\end{prop}
Non-emptiness is somewhat subtle (especially for $\ell=2$) and will occupy Paragraphs \ref{case1} and \ref{case2} below.
\begin{proof}
While \cite{paiitbgoaas} requires that the determinant be trivial, the proof as written there in Paragraph 5.1.9ff applies here, as well. Rather than repeat the details, I will use this space to give a ``guide to the literature''. Assume that $\ms M_{\ms D}^\unif(\ell^2,\ms N)$ is nonempty. To prove the result, we may thus replace $\kappa$ by $\widebar\kappa$ and assume the base field is algbraically closed. 

We can write the smooth stack $\ms M_{\ms D}^\unif(\ell^2,\ms N)$ as an ascending union of open substacks with bounded Castelnuovo-Mumford regularity; it suffices to prove that these open substacks are irreducible. We will write $M$ for one such open substack. Given the bound on the regularity, we have the following Bertini-type result. 

\begin{prop}[Proposition 5.1.12 of \cite{paiitbgoaas}]\label{sec:moduli-unif-twist-8}
  There exists a positive integer $n$ such that for any algebraically closed field $K$ containing $\kappa$ and any two objects
  $V$ and $W$ in $M(K)$, a general map $V\to W(n)$ has cokernel $Q$
  satisfying the following conditions.
  \begin{enumerate}
  \item The support $S:=\supp Q$ is a finite reduced substack of $\ms
    D^\sm$ and $Q$ is identified with a $\ms D_S$-twisted invertible
    sheaf.
  \item For any $i$ we have $|S\cap D_i|=\ell^2nH\cdot D_i$.
  \item If $D_i$ is a ghost component and $\Lambda$ is an invertible $\ms D_i$-twisted sheaf, there is a partition $$S\cap
    D_i=S_0\coprod\cdots\coprod S_{\ell-1}$$ such that (i) for each
    $j$, $|S_j|=\ell n H\cdot D_i$, and (ii)
    $Q|_{S_j}\tensor\Lambda_{S_j}^{\vee}$ is isotypic of type $j$.
  \end{enumerate}
\end{prop}
We omit the proof. The interested reader can simply read the proof written in \cite{paiitbgoaas}; it carries over word for word, with a few minor notational changes and the knowledge that the component indices $i=1,\ldots,s$ in \cite{paiitbgoaas} are reserved for what we here call ghost components.

Now the idea is to fix a single $V$ and use a space of extensions to parametrize $W(n)$, thus producing an irreducible cover of $M$. What follows is taken almost verbatim from \cite{paiitbgoaas}, starting with the paragraph preceding Proposition 5.1.13, with appropriate notational changes for the present context (and clarification of sentence structure).

For each $i$, define a $\kappa$-stack $\ms Q_i$ as follows.  The objects of
$\ms Q_i$ over $T$ are pairs $(E,L)$ with $E\subset (D_i\setminus\cup_{j\neq i}D_j)\times T$
a closed subscheme which is finite \'etale over $T$ of degree $n^2\ell
H\cdot D_i$ and $L$ an invertible $\ms D_E$-twisted sheaf. If $D_i$ is a ghost component, we fix an invertible $\ms D_i$-twisted sheaf $\Lambda_i$ and additionally require that there be a partition $E=E_0\coprod\cdots\coprod
E_{\ell-1}$ with $L|_{E_j}\tensor\Lambda_i^{\vee}$ isotypic of type
$j$.

\begin{prop}[Proposition 5.1.13 of \cite{paiitbgoaas}]\label{sec:moduli-unif-twist-9}
  For each $i$, the stack $\ms Q_i$ is irreducible.
\end{prop}
The proof of this Proposition is not entirely trivial; the interested reader can read it in \cite{paiitbgoaas} without needing additional context (aside from the mild notational differences, and the knowledge that the indices $i=1,\ldots,s$ are reserved there tthe ghost components, as above).

Finally, let $\ms Q=\prod\ms Q_i$, and let $\mf Q$ be the universal object on $\ms D\times\ms Q$. The proof of the following is almost entirely abstract nonsense and can again be read without further context in \cite{paiitbgoaas}.
\begin{lem}[Lemma 5.1.14 of \cite{paiitbgoaas}]\label{sec:moduli-unif-twist-10}
  The complex $\R(\pr_2)_\ast\rshom(\mf Q,\pr_1^\ast V)[1]$ is
  quasi-isomorphic to a locally free sheaf $\ms F$ on $\ms Q$.
  Moreover, this sheaf has the property that for any affine scheme $T$
  and any morphism $\psi:T\to\ms Q$, the set $\ms F_T(T)$ parametrizes
  extensions $0\to V\to W(n)\to Q\to 0$ with $Q$ the object of $\ms Q$
  corresponding to $\psi$.
\end{lem}

The dense open substack of $\mathbf{V}(\ms F^{\vee})$ parametrizing locally free extensions covers our moduli space $M$, showing that it is irreducible, as desired.
\end{proof}

\begin{cor}\label{C:non-emptiness-is-it}
Given an invertible sheaf $\ms N$ on $\mc D$, if the stack $\ms M_{\ms D_i}^\unif(\ell^2, \ms N|_{\mc D_i})$ is non-empty for each $i$ then there is a biregular locally free $\ms D$-twisted sheaf of rank $\ell^2$ and determinant $\ms N$.
\end{cor}
\begin{proof}
By Proposition \ref{P:uniform-irred} and the fact that $\kappa$ is PAC, it is enough to show that the hypothesis of the Corollary implies that $$\ms M_{\ms D\tensor_\kappa\widebar\kappa}^\unif(\ell^2,\ms N)\neq\emptyset.$$
By assumption, there is such a sheaf $V_i$ over each component $\ms D_i\tensor_\kappa\widebar\kappa$. Since $\ms D$ is a nodal union of the $\ms D_i$, Lemma 3.1.4.8 of \cite{twisted-moduli} shows that it is enough to produce determinant-preserving isomorphisms between the restrictions of the $V_i$ to the intersection gerbes $\ms D_i\cap\ms D_j$. But Proposition \ref{P:bireg} shows that for any such gerbe $\xi$, the restrictions $V_i|_\xi$ and $V_j|_\xi$ must be isomorphic, as there is a unique biregular $\xi$-twisted sheaf of rank $\ell^2$, and we can make the isomorphisms respect the determinant by a suitable scalar multiplication (as we are now working over $\widebar\kappa$!).
\end{proof}

\subsection{Proof of Theorem \ref{T:sheaf-on-D}}
\label{sec:existence}
We are now ready to show that biregular $\ms D$-twisted sheaves of rank $\ell^2$ exist.

By Corollary \ref{C:non-emptiness-is-it}, to prove Theorem \ref{T:sheaf-on-D}, it suffices to replace $\kappa$ by $\widebar\kappa$ and show that there are uniform biregular $\ms D_i$-twisted sheaves for each component $\ms D_i$ of $\ms D$ (now assumed to be over an algebraically closed field). The key is in the selection of the determinant sheaf $\ms N$, and it is here that the cases of odd $\ell$ and $\ell=2$ bizarrely diverge.

More precisely, the fundamental issue is caused by the fact that the regular represention of $\m_2$ has non-trivial determinant. In particular, this means that when studying uniform sheaves on ghost components for $\ell=2$, the eigensheaves (regular of rank $2$) have specific restrictions placed on their determinants by the normal bundles of the ramification components and the placement of the non-cold singular gerbes. The analysis of this crucial delicate case takes place in Paragraph \ref{case2} below.

\begin{hyp}\label{H:asses}
In this section we assume that either
\begin{enumerate}
\item $\ell$ is odd and $\ms N=\ms O$, or
\item $\ell=2$ and the ramification divisor
$D$ of $\alpha$ has the form $S+R$ as discussed in Section
\ref{sec:even-ram}, and then we let $\ms N=\ms O(S)|_{\ms D}$.
\end{enumerate}
\end{hyp}

The existence of a $\ms D_i$-twisted sheaf is a bit different for ghost and non-ghost components.

\begin{para} {\bf Existence when $D_i$ is not a ghost component}.\label{case1}
In this case, there is no eigendecomposition to contend with, and we merely seek a biregular locally free $\ms D_i$-twisted sheaf of rank $\ell^2$ and determinant $\ms N$. This construction works just as in the proof of Proposition 5.1.17 of \cite{paiitbgoaas}, where the full details are explained; we explain the essence here. 

First, it suffices to make any biregular locally free $\ms D_i$-twisted sheaf $V$ of rank $\ell^2$, since we can adjust the determinant using elementary transforms over points of $D_i^\circ = D_i\setminus\cup_{j\neq i}D_j$. More precisely, given a sheaf $V$ and an ample class $\ms O(1)$ on $D_i$, the sheaf $V(n)$ will have an ample determinant $\ms L$ such that $\ms L\tensor\ms N^{\vee}$ is isomorphic to $\ms O_D(E)$ for some $E$ that is supported entirely in $D_i^\circ$. Choosing an invertible quotient of the restriction $V(n)|_E\to L$, the kernel of the composed map $$V(n)\to\iota_\ast V(n)|_E\to\iota_\ast L$$ will have determinant $\ms N$. (An invertible twisted sheaf $L$ supported on $E$ exists since $\kappa$ is by assumption algebraically closed, and $\Br(\B\m_{\ell,\kappa})[\ell]$ is trivial by Lemma 4.1.3 of \cite{paiitbgoaas}.)

We will make a biregular twisted sheaf by formal gluing. For each gerbe $\xi\in D_i\cap\Sing(D)$, there is a unique biregular sheaf $V_\xi$ of rank $\ell^2$ (see Proposition \ref{P:bireg}). Moreover, $\ext^2(V_\xi,V_\xi)=0$, so $V_\xi$ deforms to some $V$ over a formal neighborhood $\Spec \widehat{\ms O}_{\ms D_i,\xi}$. The generic fiber is a twisted sheaf for a gerbe over $\B\m_{\ell,\Spec \kappa\(t\)}$. The Brauer group of this gerbe is  computable by the Leray spectral sequence, and it is identified with $\H^1(\Spec \kappa\(t\),\Z/\ell\Z)$ (see Proposition 4.1.4 of \cite{paiitbgoaas} for a computation of the part prime to the characteristic of $\kappa$, which is all that we need here). It follows that there is always a regular twisted sheaf of rank $\ell$, hence one of rank $\ell^2$, say $W$. Standard formal gluing results (Theorem 6.5 of \cite{MR1432058}) allow us to glue $V$ to $W$, yielding a twisted sheaf that is biregular at $\xi$. Repeating this for each such gerbe $\xi$ (and gluing to the twisted sheaf that is under construction, so that the structure is preserved at all gerbes already treated) yields the desired result.

Note that this step is independent of the parity of $\ell$ and the value of $\ms N$.
\end{para}

\begin{para} {\bf Existence when $D_i$ is a ghost component}.\label{case2}
Now the fun begins! The argument for odd $\ell$ is also written out in full detail in the proof of Proposition 5.1.17 of \cite{paiitbgoaas}. In either case, the proofs start the same way, which we recall here. We choose notation here that mostly harmonizes with Proposition 5.1.17 of \cite{paiitbgoaas}.

The ghost assumption ensures that the Brauer class of $\ms D_i\to \mc D_i$ is \emph{trivial} (since we are now working over algebraically closed $\kappa$). Thus, we can choose an invertible $\ms D_i$-twisted sheaf $\Lambda$ such that $\Lambda^{\tensor\ell}$ is the pullback of some invertible sheaf $\lambda$ on $D_i$. (The sheaf $\lambda$ is pulled back from $D_i$ and not merely $\mc D_i$ because, by Proposition \ref{P:vanishy}, we have chosen $\ms D\to\mc D$ so that the value in $\H^2(\widebar\xi,\m_\ell)$ is $0$ for all geometric gerbes in the smooth locus of $\mc D$.) Pulling back and tensoring with $\Lambda$ defines an equivalence of categories between sheaves on $\mc D_i$ and $\ms D_i$-twisted sheaves. Moreover, the regularity conditions translate directly into similar conditions for sheaves on $\mc D_i$, where the fibers are viewed directly as representations of $\m_\ell$ or $\m_\ell\times\m_\ell$.

Using this equivalence, we see that to construct a uniform $\ms D_i$-twisted sheaf, we seek to achieve the following.

\begin{goal}\label{G:goal}
Find a sheaf on $\mc D_i$ of the form 
$$V=V_0\oplus\cdots\oplus V_{\ell-1},$$
where 
\begin{itemize}
\item for each $s\in\{0,\ldots,\ell-1\}$ the sheaf $V_s$ is locally free of rank $\ell$ and isotypic of type $s$;
\item for each $\xi\in\mc D_i\cap\cup_{j\neq i}\mc D_j$, the restriction $V_s|_\xi$ is isomorphic to the representation of $\m_\ell\times\m_\ell$ that is the regular representation of the second factor tensored with the $s$-power character in the first factor (assuming that the first factor is the specialization of the generic stabilizer of $\mc D_i$);
\item the determinant of $V_s$ is isomorphic to $\lambda^{-1}\tensor\ms M$, where $\ms M$ is an invertible sheaf on $\mc D_i$ such that $\ms M^{\tensor\ell}\cong\ms N$. 
\end{itemize}
\end{goal}
The desired uniform sheaf is then $\Lambda\tensor V$. 

\begin{remark}\label{R:haha}
Note that the second and third condition together imply that the action of each stabilizer $\m_\ell\times\m_\ell$ at a point of $\mc D_i\cap\cup_{j\neq i}\mc D_j$ is via the determinant of the regular representation of the second factor (where we use the convention established in Section \ref{sec:splitting} that the first factor represents the specialization of the generic stabilizer and the second factor the specialization of the generic stabilizer of the transverse curve).
\end{remark}

{\bf When $\ell$ is odd}, this is easily arranged. The first two conditions can be dealt with using formal gluing as in the non-ghost case, and the third condition is satisfied by letting $\ms M=\ms O$ (recall: in the odd case $\ms N=\ms O$, so this does indeed give a correct root of $\ms N$) and using elementary transforms to align the determinant of $V_s$ properly (again, just as in the proof of the previous case). Since $\ell$ is odd, the regular representation has trivial determinant, so the phenomenon observed in Remark \ref{R:haha} is invisible.

{\bf When $\ell$ is even}, we need to use our understanding of the structure of $D$ exposed in Section \ref{sec:numerics}, as the determinant of the regular representation is not trivial, thus imposing concrete constraints on the determinant of each $V_s$, as in Remark \ref{R:haha}. 

First, assume that $D_i\subset S$. By Proposition \ref{P:parity} and Corollary \ref{C:ghost-det}, we have that $\ms N$ (now assumed to be $\ms O(S)$) has even degree $2d$ on $D_i$. 

Consider the stacky curve $$\Delta:=D_i\times_D\cup_{j\neq i}\mc D_j.$$ The stack $\Delta$ is a smooth Deligne-Mumford curve with coarse space $D_i$ that has $D_i\cdot\cup_{j\neq i}D_j$ stacky points, each with stabilizer $\m_2$. By Proposition \ref{P:parity}, $D_i\cdot R$ is even, and by Corollary \ref{C:ghost-dis} we know that $D_i\cdot\cup_{j\neq i}D_j=D_i\cdot R$. It follows that $\Delta$ has an even number of stacky points $\delta_1,\ldots,\delta_{2e}$, and that there is an invertible sheaf $\ms L$ on $\Delta$ such that 
\begin{enumerate}
\item $\ms L^{\tensor 2}\cong\ms N$, and 
\item for each stacky point $\iota:\B\m_2\inj\Delta$, the stabilizer acts non-trivially on the fiber $\iota^\ast\ms L$.
\end{enumerate}
Indeed, the sheaf $\ms O_{\Delta}(\delta_1+\cdots+\delta_{2e})$ has square of even degree with trivial stabilizer actions, hence is the pullback of some even-degree invertible sheaf $\ms L'$ from $D_i$. But then $\ms O(-G)|_{D_i}\tensor\ms L'$ has a square root in $\Pic(D_i)$ (since we are working over an algebraically closed field), allowing us to produce $\ms L$.

Let $N$ be an invertible sheaf on $\mc D_i$ that has non-trivial action of the generic stabilizer and satisfes $N^{\tensor 2}\cong\ms O_{\mc D_i}$. (This is possible since $D_i\cdot D_i$ is even.) Define $$V_0=\left(\ms L\tensor\lambda^{\vee}|_{\mc D_i}\right)\oplus\ms O_{\mc D_i}$$ and $$V_1=N\tensor V_0.$$
The sheaves $V_0$ and $V_1$ satisfy the requirement of Goal \ref{G:goal}, completing the proof in this case.

It remains to treat the case $D_i\subset R$. We know that such a $D_i$ is a $(-2)$-curve that meets $D$ at a single point on a component of $S$, so that $\ms N|_{D_i}\cong\ms O(1)$. Let $\Delta$ be the root construction of order $2$ applied to $\P^1$ at the point $[0:1]$, and let $\delta\subset\Delta$ be the unique stacky point, which has stabilizer $\m_2$. We know that $\ms O(1)\cong(\ms O(\delta))^{\tensor 2}$ (``$\delta$ is half of a point'').

Since $D_i$ is a $(-2)$-curve, we have that $\mc D_i\cong\B\m_2\times\Delta$. Setting 
$$V_0=\left(\ms O_\Delta(\delta)\tensor\lambda^{\vee}\right)\oplus\ms O_{\mc D_i}$$
and $$V_1=\chi\boxtimes V_0,$$
where $\chi$ is the invertible sheaf on $\B\m_2$ associated to the non-trivial character $\m_2\inj\G_m$, achieves Goal \ref{G:goal}, completing the proof.
\end{para}

\section{Formal local structures around $\Sing(D)$ over $\widebar k$}
\label{sec:formal}

In this section we lay the local groundwork for lifting a twisted
sheaf from $\ms D\tensor\widebar k$ to $\ms X\tensor\widebar k$.
The globalization will be carried out in Section \ref{sec:nonempty}.

Let $x$ be a closed point of $X\tensor\widebar k$ lying over a
singular point of $D$.  Write $A$ for the local ring $\ms
O_{X\tensor\widebar k,x}$, and let $x,y\in A$ be local equations for
the branches of $D$.  Write $A'$ for the Henselization of $A$ with
respect to the ideal $I=(xy)$; we have that $A'$ is a colimit of local
rings of smooth surfaces, each with $x$ and $y$ as regular parameters.
Finally, let $U=\spec A'\setminus Z(I)$ be the open complement of the
divisor $Z(xy)$.  

The following is an easy consequence of a fundamental result of Artin.

\begin{prop}\label{P:artin}
  Suppose $\alpha\in\Br(U)[\ell]$ has non-trivial secondary
  ramification or is unramified at $Z(x)$.  If $\ms A$ and $\ms B$ are Azumaya
  algebras on $U$ of degree $\ell$ and Brauer class $\alpha$ then $\ms A$ is isomorphic to $\ms B$.
\end{prop}
\begin{proof}
  The algebras $\ms A$ and $\ms B$ extend to maximal orders over $A$.
  The hypothesis on $\alpha$ implies that a generic division algebra
  $D$ with class $\alpha$ satisfies conditions (1.1)(ii) or (1.1)(iii)
  of \cite{MR657428}.  Maximal orders are Zariski-locally unique
  in these cases by Theorem 1.2 of \cite{MR657428}, so we conclude that
  $\ms A\cong\ms B$, as desired.
\end{proof}
\begin{cor}\label{C:sheaf-version}
  If $\ms X_U\to U$ is a $\m_\ell$-gerbe whose Brauer class $\alpha$
  satisfies the hypothesis of Proposition \ref{P:artin} then for any
  positive integer $m$ there is
  a unique $\ms X_U$-twisted sheaf of rank $\ell m$.
\end{cor}
\begin{proof}  
  By Proposition \ref{P:artin}, two locally free $\ms X_U$-twisted
  sheaves $\ms V$ and $\ms V'$ of rank $\ell$ satisfy $\send(\ms
  V)\cong\send(\ms V')$, whence there is an invertible sheaf $L$ on
  $U$ and an isomorphism $\ms V\simto\ms V'\tensor L$.  Since
  $\Pic(U)=0$, we conclude that $\ms V\cong\ms V'$.

  On the other hand, given a locally free $\ms X_U$-twisted sheaf $\ms
  W$ of rank $\ell m$, we claim that $\ms W$ admits a locally free
  quotient of rank $\ell$.  To see this, note that $U$ is a Dedekind
  scheme and thus any torsion free sheaf is locally free.
  Furthermore, $\alpha$ has period $\ell$ and therefore index $\ell$
  by de Jong's theorem \cite{MR2060023}.  Thus, any torsion free
  quotient of $\ms W$ of rank $\ell$ is a locally free quotient.  As a
  consequence, we can write $\ms W$ as an extension $0\to \ms K\to\ms
  W\to\ms V\to 0$ with $\ms K$ of rank $\ell(m-1)$.  By induction we
  know that $\ms K\cong\ms V^{\oplus m-1}$.  To establish the claim it
  thus suffices to show that $\ext^1_{\ms X_U}(\ms V,\ms V)=0$.  Since
  both are $\ms X_U$-twisted, we see that $\ext^1_{\ms X_U}(\ms V,\ms
  V)=\H^1(U,\send(\ms V,\ms V))$, so it suffices to show that for any
  locally free sheaf $\ms T$ on $U$ we have $\H^1(U,T)=0$.  But $U$ is
  the complement of the vanishing of a single element of $A$, so it is
  affine.  Thus, all higher cohomology of coherent sheaves vanishes.
\end{proof}

The key consequence of this formal statement is a Zariski-local
existence statement.  Assume that $k$ is
algebraically closed and let $q\in D$ be a singular point.  Write
$\xi$ for the closed residual gerbe of $\ms X$ lying over $q$ and $\ms
X_q$ for the fiber product $\ms X\times_X\spec\ms O_{X,q}$.  Finally,
write $\ms X_U$ for $\ms X\times_X(\spec\ms O_{X,q}\setminus D)$.

\begin{prop}\label{P:local-existence}
Given a locally free $\ms X_U$-twisted sheaf $\ms V_U$ of rank $\ell^2$ and a
regular locally free $\xi$-twisted sheaf $\ms V_\xi$ of rank $\ell^2$, there is a locally
free $\ms X_q$-twisted sheaf $\ms V$ of rank $\ell^2$ such that $\ms
V|_U\cong\ms V_U$ and $\ms V|_\xi\cong\ms V_\xi$.
\end{prop}
\begin{proof}
  Write $\widehat{\ms D}=\ms D\times_X\spec\widehat{\ms O}_{X,q}$.
  Basic deformation theory shows that $\ms V_\xi$ deforms to a locally
  free $\widehat{\ms D}$-twisted sheaf $\widehat{\ms V}_\xi$ of rank
  $\ell^2$.  On the other hand, if $\ms D_{\eta_i}$ denotes the
  restriction of $\ms D$ to a generic point of $D$ (in the local
  scheme $\spec\ms O_{X,q}$), there is a unique regular $\ms
  D_{\eta_i}$-twisted sheaf $\ms R_i$ of rank $\ell^2$ by
  Lemma \ref{L:regzar}.  Thus, $\ms R_i|_{\widehat{\ms D}\times_{\ms D}\ms
    D_{\eta_i}}\cong\widehat{\ms V}_{\xi}|_{\widehat{\ms D}\times_{\ms
      D}\ms D_{\eta_i}}$.  Applying Theorem 6.5 of \cite{MR1432058},
  we see that there is a locally free $\ms D$-twisted sheaf
  $\widetilde{\ms V}$ such that $\widetilde{\ms V}|_{\xi}\cong\ms
  V_\xi$.

    Now apply the same argument again.  The same result shows that
    $\widetilde{\ms V}$ deforms to a $\ms X\times_X\spec A'$-twisted
    sheaf $\ms W$ of rank $\ell^2$.  Since $\ms V_U|_{\spec
      A'}\cong\ms W|_U$, we can again apply Theorem 6.5 of \cite{MR1432058} to
    conclude that there is a $\ms V$ as claimed in the statement.
\end{proof}

\section{Extending quotients}
\label{sec:elem-transf-10}

The results of this section are the second component (in addition to
the local analysis of Section \ref{sec:formal}) needed in Section \ref{sec:nonempty}
to solve the problem of lifting a $\ms D$-twisted sheaf to an $\ms
X$-twisted sheaf.  To start, we recall the notion of elementary
transformation.

\begin{defn}
  Let $i:Z\subset Y$ be a divisor in a regular Artin stack.  Given a
  locally free sheaf $V$ on $Y$ and an invertible quotient $i^\ast V\surj
  Q$, the \emph{elementary transformation of $V$ along $Q$\/} is the
  kernel of the induced map $V\surj i_\ast Q$.
\end{defn}

It is a basic fact that the elementary
transformation of $V$ along $Q$ has determinant isomorphic to
$\det(V)(-Y)$.  This is proven in Appendix A of \cite{period-index-paper}.

Call an Artin stack \emph{Dedekind\/} if it is Noetherian and regular
and each connected component has dimension $1$.

\begin{lem}\label{sec:elem-transf-7}
  Let $C$ be a connected tame separated Dedekind stack
  with trivial generic stabilizer with a coarse moduli space
  $C\to\widebar C$.  Given a finite closed substack $S\subset C$ and a
  locally free sheaf $W_S$ of rank $r$ on $S$, there is a locally free
  sheaf $W$ on $C$ and an isomorphism $W|_S\simto W_S$.
\end{lem}
\begin{proof}
  Let $\widebar S\subset\widebar C$ be the reduced image of $S$ in
  $\widebar C$.  Since $C$ is tame and proper over $\widebar C$,
  infinitesimal deformation theory and the Grothendieck Existence
  Theorem for stacks (Theorem 1.4 of \cite{MR2183251}) show that $W_S$ is the restriction of a locally
  free sheaf $\widehat W$ of rank $r$ on $C\times_{\widebar
    C}{\spec\widehat{\ms O}_{\widebar C,\widebar S}}$, the semilocal
  completion of $C$ at $S$.  Let $U=\spec\ms O_{\widebar C,\widebar
    S}\setminus\widebar S$, and let $\widehat U=U\times_{\spec\ms
    O_{\widebar C,\widebar S}}\spec\widehat{\ms O}_{\widebar
    C,\widebar S}$.  Since locally free sheaves of rank $r$ over
  fields are unique up to isomorphism, we have that given a locally
  free sheaf $W_U$ of rank $r$ on $U$, there is an isomorphism
  $W_U|_{\widehat U}\simto \widehat W|_{\widehat U}$.  Applying
  Theorem 6.5 of \cite{MR1432058}, we see that there is a locally free sheaf $W'$ of
  rank $r$ on $C\times_{\widebar C}\spec \ms O_{\widebar C,\widebar
    S}$.  Since $\ms O_{\widebar C,\widebar S}$ is a limit of open
  subschemes of $\widebar C$ with affine transition maps and $W'$ is
  of finite presentation, we see that there is an open substack
  $V\subset C$ containing $S$ and a locally free sheaf $W_V$ of rank
  $r$ such that $W_V|_S$ is isomorphic to $W_S$.  Taking any torsion
  free (and thus locally free) extension of $W_V$ to all of $C$ yields
  the result.
\end{proof}

\begin{lem}\label{sec:elem-transf-8}
  Let $C$ be a connected separated Dedekind stack with trivial generic
  stabilizer and coarse moduli space $C\to\widebar C$.  Let $V$ be a
  locally free $\ms O_C$-module.  Given a finite closed substack
  $S\subset C$ and a locally free quotient $V|_S\surj Q_S$, there is a
  locally free quotient $V\surj Q$ whose restriction to $S$ is $V|_S\to Q_S$.
\end{lem}
\begin{proof}
  Let $K_S\subset V|_S$ be the kernel of $V|_S\to Q_S$.  By Lemma
  \ref{sec:elem-transf-7} there is a locally free sheaf $K$ on $C$ and
  an isomorphism $K|_S\simto K_S$.  Since $C$ is tame, the map
  $\hom_{\spec\ms O_{C,S}}(K,V)\to\hom_S(K_S,V_S)$ is surjective, and
  thus there is a map $K_{\spec\ms O_{C,S}}\inj V_{\spec\ms O_{C,S}}$
  with cokernel $Q'$ restricting to $Q_S$ over $S$.  Since $\spec\ms
  O_{C,S}$ is a limit of open substacks with affine transition maps
  and everything is of finite presentation, we see that there is an
  open substack $U\subset C$ and an extension $V|_U\to Q_U$ as
  desired.  Taking the saturation of the kernel of $V|_U\to Q_U$ in
  $V$ yields the result.
\end{proof}

\section{Lifting $\ms D$-twisted sheaves to $\ms X$-twisted sheaves
  over $\widebar k$}
\label{sec:nonempty}

In this section we prove a result that should be viewed as
a non-commutative analogue of the classical statement that a vector bundle
on a smooth curve in a projective surface whose determinant extends to the
ambient surface itself extends to the surface.

\begin{prop}\label{P:nonempt}
  Let $\ms W$ be a regular locally free $\ms D$-twisted sheaf of rank
  $\ell^2$ and determinant that extends from $D$ to $X$.  There is a locally free $\ms
  X\tensor\widebar k$-twisted sheaf $\ms V$ of rank $\ell^2$ and
  trivial determinant such that $\ms V|_{\ms D}\cong\ms W$.
\end{prop}

To prove Proposition \ref{P:nonempt} we may assume that $k$ is
algebraically closed.  To start, the local results of Section \ref{sec:formal} immediately
give us the following.  We keep $\ms W$ fixed throughout this section.

\begin{prop}
  There is a locally free $\ms X$-twisted sheaf $\ms V$ of trivial
  determinant such that $\ms V|_{\ms D}$ is a Zariski form of $\ms W$.
\end{prop}
\begin{proof}
  By de Jong's theorem (the main result of \cite{MR2060023}), there is a $\ms X_{X\setminus D}$-twisted
  sheaf $\ms V_0$ of rank $\ell^2$, which we fix.  Let $\xi\in\ms D$ be a
  singular residual gerbe with image $q\in X$.  By Proposition \ref{P:local-existence},
  there is a locally free $\ms X\times_X\spec \ms O_{X,q}$-twisted
  sheaf $\ms V_\xi$ of rank $\ell^2$ such that $\ms
  V_{\xi}|_\xi\cong\ms W|_{\xi}$ and $\ms V_\xi|_{X\setminus
    D}\cong\ms V_0|_{\spec \ms O_{X,q}}$.  Zariski gluing then extends
  $\ms V_0$ over $\xi$ so that its restriction to $\xi$ is isomorphic
  to $\ms W|_{\xi}$.  By induction on the number of singular points of
  $D$, we conclude that there is an open subscheme $X_0\subset X$
  containing the singular points of $D$ and a locally free $\ms
  X\times_XX_0$-twisted sheaf $\ms V^0$ such that $\ms
  V^0|_\xi\cong\ms W|_\xi$ for each singular residual gerbe $\xi$ of
  $\ms D$.  Taking a reflexive hull of $\ms V^0$ yields a locally free
  $\ms X$-twisted sheaf $\ms V$ with the same local property at each
  $\xi$.

  We claim that $\ms V|_{\ms D}$ is a form of $\ms W$.  To see this,
  note that by Nakayama's lemma this is true in a neighborhood of each
  singular point $q\in D$.  On the other hand, on the smooth locus of
  $\ms D$ any two regular twisted sheaves of the same rank are Zariski
  forms of one another by Lemma \ref{L:regzar}.

It remains to ensure that $\ms V$ has trivial determinant.  To do
this, we may assume after twisting $\ms V$ by a suitable power of $\ms
O(1)$ that $\det\ms V\cong\ms O(C)$ with $C\subset\mc X$ a smooth divisor
meeting $\mc D$ transversely.  By Tsen's theorem, $\ms X|_C$ has
trivial associated Brauer class, so $\ms V|_C$ has invertible
quotients.  Taking the elementary transformation along any such quotient
$\ms V\to\ms Q$ yields a subsheaf $\ms V'\subset\ms V$ with trivial
determinant which is isomorphic to $\ms V$ at each singular residual
gerbe $\xi\in\ms D$, as desired.
\end{proof}

\begin{proof}[Proof of Proposition \ref{P:nonempt}]
Since $\ms V|_{\ms D}$ is a form of $\ms W$, for all sufficiently
large $N$ we can recover $\ms W$ as the kernel of a surjection $\ms
V(N)\to Q$ with $Q$ a reduced $\ms X$-twisted sheaf of dimension $0$
with support equal to $C\cap\mc D$ for a general smooth $C\subset\mc
X$ (belonging to the linear system $|\ms O(\ell^2N)|$) meeting $\mc D$ transversely.  The following Lemma enables us to
lift the elementary transformation to $\ms X$.

\begin{lem}\label{L:lift-transform}
  Let $C\subset\mc X$ be a smooth divisor meeting $\mc D$ transversely
  with preimage $\ms C\subset\ms X$.  Given an invertible quotient
  $\chi:\ms V|_{\ms D}\surj Q$ defined over $C\cap\mc D$, there is an
  invertible quotient $\ms V\to\ms Q$ defined over $\ms C$
  extending $\chi$.
\end{lem}
\begin{proof}
  Choose an invertible $\ms C$-twisted sheaf $L$ and let $V=\ms V_{\ms
    C}|\tensor L^\vee$ and $\widebar Q=Q\tensor L^\vee$.  By abuse of
  notation, $V$ is a sheaf on the smooth tame Dedekind stack $C$,
  which has a trivial generic stabilizer, and $\widebar Q$ is an
  invertible quotient of $V|_{C\cap \mc D}$.  By Lemma
  \ref{sec:elem-transf-8}, there is an invertible quotient
  $V\to\widetilde Q$ whose restriction to $C\cap\mc D$ is $\widebar
  Q$.   Twisting up by $L$ yields a quotient $\ms V|_{\ms C}\to\ms Q$
  extending the given quotient $\ms V|_{\ms C}\to Q$.  This yields the
  quotient extending $\chi$, as desired.
\end{proof}
Since we can realize $\ms W$ as an elementary transformation of $\ms
V(N)|_{\ms D}$ along an invertible sheaf on $C\cap\ms D$, Lemma
\ref{L:lift-transform} produces an elementary transformation of $\ms
V(N)$ whose restriction to $\ms D$ is $\ms W$ and whose determinant is
trivial, giving a locally free $\ms X$-twisted sheaf of trivial
determinant lifting $\ms W$, as desired.
\end{proof}

\section{Proof of Theorem \ref{T:part}}
\label{sec:irred}

In this section we prove Theorem \ref{T:part}.  The method used is a
fundamental idea that recurs throughout many moduli problems, notably
in the study of moduli of sheaves by O'Grady \cite{MR1376250} and twisted sheaves
by the author \cite{twisted-moduli}, and in the recent work of de Jong, He, and Starr on
rational sections of fibrations over surfaces \cite{starr}.

Let $\Xi$ be the set of connected components of $\ms U\tensor\widebar
k$ and $\Xi(c)$ the set of connected components parametrizing $\ms V$
such that $c(\ms V)=c$.  There is a natural action of $\Gal(\widebar
k/k)$ on $\Xi$ preserving each $\Xi(c)$, so that the Chern class $c$
induces a $\Gal(\widebar k/k)$-equivariant map $\widebar c:\Xi\to\Z$
(where the target has the trivial action).  

\begin{lem}\label{L:Gal-orbits-finite}
  The orbits of $\Xi$ under the action of $\Gal(\widebar k/k)$ are finite.
\end{lem}
\begin{proof}
  The Galois action on $\ms U\tensor\widebar k$ preserves $\ms
  U(c,N)\tensor\widebar k$, which is of finite type.  It is elementary
  that there is an open normal subgroup $H_{c,N}\subset\Gal(\widebar
  k/k)$ acting trivially on the set of components of $\ms U(c,N)$.
  Given a connected component $\ms U_0\subset\ms U\tensor\widebar k$,
  any point $\gamma\in\ms U_0$ lies in $\ms U(c,N)$ for some $c,N$, so
  that there is a component $\ms U(c,N)_0$ containing $c$.  If $h\in
  H_{c,N}$ then $h$ sends $\ms U(c,N)_0$ to itself and thus sends $\ms
  U_0$ to a connected component whose intersection with $\ms U(c,N)_0$
  is $\ms U(c,N)_0$. Since all of the stacks in question are smooth,
  any two connected components that intersect are equal, which implies
  that $H_{c,N}$ stabilizes $\ms U_0$.  Thus, $\ms U_0$ has finite
  orbit.
\end{proof}

The main idea in the proof of Theorem \ref{T:part} is the following.  Let $\ms N$ be either $\ms O$ if $\ell$ is odd or $\ms O(S)$ if $\ell=2$.  
Suppose $\ms W$ is a locally free $\ms D$-twisted
sheaf of rank $\ell^2$ and determinant $\ms N$ (see Section \ref{sec:existence-d}).  The usual
calculations in deformation theory show that $\ms M_{\ms D}(\ell^2,\ms
N)$ is a smooth stack over the base so that $\ms W$ defines
a geometrically integral connected component $\ms M_{\ms D}(\ms W)$.
(Note: this holds even when $\ms D$ is not geometrically connected
over $k$!)

Restriction defines a morphism $\operatorname{res}:\ms U\to\ms
M_{\ms D}(\ell^2,\ms N)$.

\begin{notn}
  Write $\ms U(\ms W)$ for the preimage of the open substack
  $\ms M_{\ms D}(\ms W)$ under the restriction morphism $\operatorname{res}$ described above.  Denote the
  set of connected components of $\ms U(\ms W)$ by $\Xi(\ms W)$.
\end{notn}
Since $\ms U$ is smooth (but not separated!), the inclusion $\ms U(\ms
W)$ induces an injective Galois-equivariant morphism $\Xi(\ms
W)\inj\Xi$.  Lemma \ref{L:Gal-orbits-finite} implies that the Galois
orbits of $\Xi(\ms W)$ are therefore finite.  We will write $\Xi(\ms
W)(c)=\Xi(\ms W)\cap\Xi(c)$.

Before stating the main result of this section, we require one more
definition.

\begin{defn}
  Call a sequence of elements $x_1,x_2,\ldots,x_n\in\Xi$
  \emph{equisingular\/} if there is a nonnegative integer $m$ and
  sheaves $\ms V_i\in x_i$ such that for all $i$ the sheaf $\ms
  V_i^{\vee\vee}/\ms V_i$ is isomorphic to the pushforward of an
  invertible $\ms X\times_XS_i$-twisted sheaf to $\ms X$, where $S_i$ is a finite
  closed subscheme of $X\setminus D$ of length $m$.
\end{defn}

In particular, an equisingular sequence of length $1$ corresponds to a
component containing a sheaf $\ms V$ such that $\ms V^{\vee\vee}/\ms
V$ is a direct sum of invertible twisted sheaves supported on closed
residual gerbes of $\ms X\setminus\ms D$.

\begin{remark}\label{R:ok}
  The argument of Lemma \ref{L:U-nonempty} applied to Proposition
  \ref{P:nonempt} shows that there is an equisingular element of
  $\Xi(\ms W)$.
\end{remark}

\begin{prop}\label{P:contract}
  There is a $\Gal(\widebar k/k)$-equivariant map $\tau:\Xi(\ms
  W)\to\Xi(\ms W)$ such that for any $c$ and any equisingular sequence 
  $$x_1,x_2\in\Xi(\ms W)(c)$$ there is a natural number $n$ such that $$\tau^{\circ
    n}(x_1)=\tau^{\circ n}(x_2).$$
\end{prop}

Before producing $\tau$, let us show how Proposition \ref{P:contract}
proves Theorem \ref{T:part}.

\begin{proof}[Proof of Theorem \ref{T:part} using Proposition \ref{P:contract}]
  Let $x\in\Xi(\ms W)$ be any equisingular element (see Remark \ref{R:ok}).  By Lemma
  \ref{L:Gal-orbits-finite}, the Galois orbit of $x$ is finite, say
  $x=x_1,x_2,\ldots,x_m$, and is entirely contained in $\Xi(\ms
  W)(\widebar c(x))$.  Moreover, $x_1,x_2,\ldots,x_m$ are equisingular
  (as one can see by applying the Galois action to a sheaf
  representing $x$).  By Proposition \ref{P:contract}, there is an
  element $y\in\Xi$ and an iterate $\tau'$ of $\tau$ such that
  $\tau'(x_i)=y$ for all $i=1,\ldots,m$.  For any $g\in\Gal(\widebar
  k/k)$, we have that $g\cdot y=g\cdot\tau'(x)=\tau'(g\cdot x)=y$, so
  that $y$ is Galois-invariant.  But then $y$ corresponds to a
  geometrically integral component $\ms S\subset\ms U(\ms W)$, as
  desired.  Indeed, let $\ms I\subset\ms O_{\ms U(\ms
    W)\tensor\widebar k}$ be the ideal sheaf of the component $\ms
  U_y\subset\ms U(\ms W)\tensor\widebar k$ corresponding to $y$.
  Choose a smooth cover $U\to\ms U(\ms W)$ and let $\ms I'=\ms
  I\tensor\ms O_{U\tensor\widebar k}$.  Since $y$ is Galois-fixed we
  have that $\ms I'$ is preserved by the canonical descent datum on
  $\ms O_{U\tensor\widebar k}$ induced by the extension
  $k\subset\widebar k$.  Descent theory for schemes shows that $\ms
  I'$ is the base change of a sheaf of ideals $\ms J\subset\ms O_U$
  cutting out an open subscheme $U_0\subset U$.  Since
  $U_0\tensor\widebar k$ is equal to the preimage of its image in $\ms
  U(\ms W)\tensor\widebar k$, we conclude the same about $U_0$, which
  therefore corresponds to an open subscheme $\ms S\subset\ms U(\ms
  W)$ such that $\ms S\tensor\widebar k=Z(\ms I)$.
\end{proof}

It remains to prove Proposition \ref{P:contract}.  The map $\tau$ is
defined as follows.  

\begin{construction}\label{cons:tau}
  Given a component $y$ of $\ms U(\ms W)\tensor\widebar k$
  corresponding to a locally free $\ms X\tensor\widebar k$-twisted
  sheaf $\ms V$ of rank $\ell^2$ and trivial determinant lying in $\ms
  U(\ms W)$, define a new sheaf $\ms V'$ by choosing a point $x\in
  X(\widebar k)\setminus D(\widebar k)$ around which $\ms V$ is
  locally free and forming an exact sequence
$$0\to\ms V'\to\ms V\to L_x\to 0,$$
where $L_x$ is a locally free $\ms X\times_X x$-twisted sheaf of rank
$1$ and $\ms V\to L_x$ is a surjection.  Since $\ms V'|_{\ms D}$ is
isomorphic to $\ms V|_{\ms D}$, the sheaf $\ms V'$ determines a new
component $\tau(y)\in\Xi(\ms W)$.
\end{construction}
\begin{lem}\label{L:well-defd}
  Construction \ref{cons:tau} is well-defined and Galois-equivariant.
\end{lem}
\begin{proof}
  Let $O\subset X\setminus D$ be the open subscheme over which $\ms V$
  is locally free.  Since the family of invertible quotients of the
  restriction of $\ms V$ to a point $x\in O$ is connected, we see that
  all quotients $\ms V'$ arising as in Construction \ref{cons:tau} lie
  in a connected family.  On the other hand, since $\ms V$ is
  unobstructed so is $\ms V'$, and this implies that any two objects
  lying in a connected family must lie in the same connected
  component.

  Galois-equivariance of $\tau$ follows from the argument of the
  preceding paragraph, along with the fact that the Galois group sends
  a pair $(x,\ms V\surj L_x)$ with $x\in O$ to another such pair.
\end{proof}

\begin{remark}
  As a consequence of Lemma \ref{L:well-defd}, we can compute the
  $m$th iterate of $\tau$ by taking an invertible quotient over a
  finite reduced subscheme of $O$ of length $m$ (where $O$ still
  denotes the locus over which $\ms V$ is locally free).
\end{remark}

It remains to verify that $\tau$ is a contracting map (in the weak
sense enunciated in Proposition \ref{P:contract}).  Since $x_1$ and
$x_2$ are equisingular, we can choose $\ms V_i\in x_i$, $i=1,2$, such
that $\ms V_i^{\vee\vee}/\ms V_i$ is supported at $m$ closed residual
gerbes.  Suppose $\supp(\ms V_1^{\vee\vee}/\ms V_1)\cap\supp(\ms
V_2^{\vee\vee}/\ms V_2)$ has $m'$ closed residual gerbes.  Applying
$\tau^{\circ m-m'}$ to $x_1$ and $x_2$ we can assume that $\ms V_1$
and $\ms V_2$ are everywhere Zariski-locally isomorphic (by taking
quotients of $\ms V_1$ along $\supp(\ms V_2^{\vee\vee}/\ms
V_2)\setminus\supp(\ms V_1^{\vee\vee}/\ms V_1)$ and similarly for $\ms
V_2$).  We are thus reduced to the following.

\begin{prop}\label{P:contracting}
  Suppose $\ms V_1$ and $\ms V_2$ are two torsion free $\ms
  X\tensor\widebar k$-twisted sheaves of rank $\ell^2$ and trivial
  determinant belonging to $\ms U(\ms W)(c)$ which are everywhere
  Zariski-locally isomorphic.   Then there are coherent
  subsheaves $\ms V_i'\subset\ms V_i$, $i=1,2$, such that
  \begin{enumerate}
  \item $\ms V_i/\ms V_i'$ is reduced and supported over $m$ closed
    points of $X\setminus D$, with $m$ independent of $i$;
  \item there is a connected $\widebar k$-scheme $T$ containing two
    points $[1],[2]\in T(\widebar k)$ and a morphism $$\omega:T\to\ms
    U(\ms W)$$ such that $\omega([i])\cong [\ms V_i']$ for $i=1,2$.
  \end{enumerate}
  In other words, $\ms V_1'$ and $\ms V_2'$ give the same element of
  $\Xi(\ms W)(c+md)$, where $d=\deg u$.
\end{prop}
\begin{proof}
  The proof is very similar to the proof in Paragraph
  3.2.4.19 of \cite{twisted-moduli}.  We present it using a series of
  lemmas.

  \begin{lem}\label{L:sliding-cok}
    Suppose $\ms E\subset\ms X$ is an effective Cartier divisor and
    $\ms G\subset\ms X$ is a non-empty open substack.  Given a torsion
    free $\ms X$-twisted sheaf $\ms F$ of rank $r$ prime to $p$ and
    trivial determinant such that $\ms F|_{\ms E}$ is locally free,
    there exists a coherent subsheaf $\ms F'\subset\ms F$ such that
    \begin{enumerate}
    \item the sheaf $\ms F'$ is unobstructed;
    \item the quotient $\ms F/\ms F'$ is reduced and $0$-dimensional
      with support contained in $\ms G$;
    \item the restriction map on equideterminantal miniversal 
      deformation spaces $$\Def_0(\ms F)\to\Def_0(\ms F|_{\ms E})$$ is surjective.
    \end{enumerate}
  \end{lem}
  \begin{proof}
    Let $\ms L$ be an invertible sheaf on $\ms X$.  We claim that
    there is a subsheaf of the required type $\ms F'\subset\ms F$ such
    that $\ext^2_0(\ms F'\tensor\ms L,\ms F')=0$.  To see this, note
    first that by Serre duality and the hypothesis that $\ell$ is
    prime to $p$ we know that $\ext^2_0(\ms
    F\tensor\ms L,\ms F)$ is dual to $\hom_0(\ms F,\ms F\tensor\ms
    L\tensor K_{\ms X})$ (and similarly for $\ms F'$), so it suffices
    to prove that one can make $\hom_0(\ms F',\ms F'\tensor\ms L)$
    vanish (replacing $\ms L\tensor K$ by $\ms L$).  In addition, note
    that when $\ms F/\ms F'$ has finite support in the locally free
    locus of $\ms F$, there is a canonical inclusion $$\hom_0(\ms
    F',\ms F'\tensor\ms L)\inj\hom_0(\ms F,\ms F\tensor\ms L)$$
    identifying the former with the space of homomorphisms which
    preserve (in fibers) the kernel of the induced quotient map $\ms
    F_{\supp \ms F/\ms F'}\to\ms F/\ms F'$.  (This last inclusion is
    produced by realizing $\ms F$ locally as the reflexive hull of
    $\ms F'$, where they differ.)

    Since the homomorphisms in question are traceless, they cannot
    preserve all codimension $1$ subspaces of a general geometric
    fiber.  Thus, for a general point $x\in\ms G$ and a general
    reduced quotient $$\ms F\surj\ms F_x\surj Q$$ supported at $x$ with
    kernel $\ms F'$, the inclusion $$\hom_0(\ms F',\ms F'\tensor\ms
    L)\inj\hom_0(\ms F,\ms F\tensor\ms L)$$ is not surjective.  By
    induction on $\dim\hom_0(\ms F,\ms F\tensor\ms L)$ we can find a
    sequence of such subsheaves for which the associated $\hom_0$ is
    trivial, as desired.

    Now, given a sheaf $\ms F$ locally free around $\ms E$, the
    tangent map $$\Def_0(\ms F)\to\Def_0(\ms F|_{\ms E})$$ is given by
    the restriction map $$(\ext_{\ms X}^1)_0(\ms F,\ms
    F)\to(\ext^1_{\ms E})_0(\ms F|_{\ms E},\ms F|_{\ms E}),$$ which by
    the cher-\`a-Cartan isomorphism is canonically isomorphic to the
    restriction map $$\ext^1_0(\ms F,\ms F)\to\ext^1_0(\ms F,\ms
    F|_{\ms E})$$ (with both $\ext$ spaces on $\ms X$).  The cokernel
    of this map is contained in $\ext^2_0(\ms F,\ms F(-\ms E))$, and
    by the first two paragraphs of this proof we can find $\ms
    F'\subset\ms F$ of the desired form so that $\ext^2_0(\ms F,\ms
    F(-\ms E))=0$.  Taking a further subsheaf, we may also assume that
    $\ext^2_0(\ms F',\ms F')=0$, so that $\ms F'$ is unobstructed, as
    desired.
  \end{proof}
  Given $\ms V_1,\ms V_2\in\ms U(\ms W)$, we can thus find
  (unobstructed) subsheaves $\ms V_i'\subset\ms V_i$ such that the
  restriction morphism $\ms U(\ms W)\to\ms M(\ms W)$ is dominant at
  $\ms V_i'$ for $i=1,2$.  Deforming $\ms W$ to the generic member of
  $\ms M(\ms W)$ and following by deformations of $\ms V_i'$, we may
  thus assume that $\ms V_1'|_{\ms D}\cong\ms V_2'|_{\ms D}$.  Taking
  further subsheaves if necessary, we may assume that for each
  geometric point $x\to X$, the strict Henselizations $\ms
  V_i'|_{\spec\ms O_{X,x}^{sh}}$ are isomorphic.  We will relabel $\ms
    V_i'$ by $\ms V_i$ (acknowledging that we have already started
    iterating $\tau$ on the original components).

    By Proposition \ref{L:cok}, for sufficiently large $N$, the cokernel $\ms Q$ of a
    general map $\ms V_1\to\ms V_i(N)$, $i=1,2$, is an invertible $\ms
    X$-twisted sheaf supported on the preimage of a smooth curve $C$ in $X$ in
    the linear system $|\ell^2NH|$ meeting $D$ transversely. In particular,
    there exists one such curve $C$ and two invertible $C\times_X\ms X$-twisted
    sheaves $\ms Q_1$ and $\ms Q_2$ such that there are extensions $$0\to\ms
    V_1\to\ms V_i(N)\to\ms Q_i\to 0$$ for $i=1,2$.

  \begin{remark}\label{R:ext}
  Choosing isomorphisms $\ms W\simto\ms V_i|_{\ms D}$, we may assume
  (since $N$ is allowed to be arbitrarily large) that each extension
  has the same restriction to an extension
$$0\to\ms W\to\ms W(N)\to Q|_{\ms D}\to 0$$
of sheaves on $\ms D$.
\end{remark}

Write $\ms C:=C\times_X\ms X$ and let $\iota:\ms C\to\ms X$ be the
canonical inclusion map.
  \begin{lem}
    There is an irreducible $k$-scheme $T$ with two $k$-points $[1]$
    and $[2]$ and an invertible $\ms C\times T$-twisted sheaf $\mf Q$
    such that $\mf Q_{[i]}\cong\ms Q_i$ for $i=1,2$.
  \end{lem}
  \begin{proof}
    By Remark \ref{R:ext}, we know that $\ms Q_1|_{\ms D}\cong\ms
    Q_2|_{\ms D}$.  Furthermore, we have the equality $$[\ms
    Q_i]=[\LL\iota^\ast\ms V_i(N)]-[\LL\iota^\ast\ms V_1]$$ in $K(\ms
    C)$.  Since $c(\ms V_1)=c(\ms V_2)$ and $\det\ms V_1\cong\det\ms
    V_2$, we conclude that $u^\ast\ms Q_1$ has the same Hilbert
    polynomial as $u^\ast\ms Q_2$.

    Thus, we find that $\ms Q_1$ and $\ms Q_2$ are two invertible
    sheaves on $\ms C$ with the same degree when pulled back to the
    curve $Z\times_{\mc X}C$ and with isomorphic restrictions to every
    residual gerbe of $\ms C$.  The sheaf $\ms Q_1\tensor\ms Q_2^\vee$
    is thus the pullback of an invertible sheaf $\Gamma$ of degree $0$
    on the coarse moduli space $\widebar C$ of $C$ (which is the
    coarse moduli space of $\ms C$).  Since $C$ intersects $\ms D$
    transversely, $\widebar C$ is a smooth curve in $X$.  Let $G$ be a
    tautological invertible sheaf over $\widebar C\times
    \Pic^0_{\widebar C/k}$, and write $[1]$ for the point
    corresponding to the trivial invertible sheaf and $[2]$ for the
    point parametrizing $\Gamma$.  The sheaf $$G_{\ms
      C\times\Pic^0_{\widebar C/k}}\tensor(\ms Q_1)_{\ms
      C\times\Pic^0_{\widebar C/k}}$$ on $\ms C\times\Pic^0_{\widebar
      C/k}$ gives the desired irreducible interpolation.
  \end{proof}
  The end of the proof of Proposition \ref{P:contracting} is very
  similar to the proof of Proposition 3.2.4.22 in \cite{twisted-moduli}.  By cohomology
  and base change, for sufficiently large $m$ the vector spaces
  $\ext^1(\mf Q_t(-m),\ms V_1)$ form a vector bundle $\V$ on $T$ such
  that there is a universal extension $$0\to(\ms V_1)_T\to\ms V\to\mf
  Q(-m)\to 0$$ over $\ms X\times T$.  Let $\V^\circ\subset \V$ be the
  open subset over which $\ms V$ has unobstructed torsion free fibers.
  For each $i=1,2$, choosing a general section of $\ms O(-m)$ and
  forming the pullback
$$\xymatrix{0\ar[r]&\ms V_1\ar[r]\ar[d]&\ms V_i(N)'\ar[d]\ar[r]&\ms Q_i(-m)\ar[r]\ar[d]& 0\\
  0\ar[r]&\ms V_1\ar[r]&\ms V_i(N)\ar[r]&\ms Q_i(-m)\ar[r]& 0}$$
yields a subsheaf $\ms V_i(N)'$ of $\ms V_i(N)$ such that the quotient
$\ms V_i(N)/\ms V_i(N)'$ is the pushforward of an invertible twisted
sheaf supported on finitely many closed residual gerbes of $\ms
C\setminus\ms D$.  Thus, the sheaf $\ms V(-N)$ contains two fibers
over $\V^\circ$ parametrizing the finite colength subsheaves $\ms
V_i(N)'(-N)\subset\ms V_i$, as desired.  This completes the proof of
Proposition \ref{P:contracting}.
\end{proof}

\appendix
\section{A Bertini theorem}
In this appendix we record a simple Bertini type result for general
maps between sheaves on stacks of the kind encountered in this paper. We will
study when a general map of the form $V\to W(N)$ has a nice cokernel
(one that is invertible or an invertible sheaf supported on a divisor). The main
restriction that is not apparent in the classical theorems is the condition that
the sheaves $V$ and $W$ must be locally isomorphic everywhere, so that
the local maps between them are not forced to vanish somewhere by pure
representation theory.

We retain the notation from Sections \ref{sec:intro} through \ref{sec:gerbe}, so $\ms X$ is a $\m_\ell$-gerbe on a stack that arises from applying the root construction to a surface $X$ along components of an snc divisor $D$. Fix an ample divisor $H$ on $X$. Let $V$ and $W$ be torsion free regular $\ms X$-twisted sheaves of rank $\ell^2$ that are everywhere Zariski-locally isomorphic. More general statements are undoubtedly true, but our goal is not to maximize generality at the expense of utility.

  \begin{prop}\label{L:cok}
    For sufficiently large $N$, the cokernel $\ms Q$ of a general map $$V\to W(N)$$ is an invertible $\ms X$-twisted sheaf
    supported on the preimage of a smooth curve $C$ in $X$ in the linear
    system $$|rNH + \det(W) - \det(V)|$$ meeting $D$ transversely.
  \end{prop}
  \begin{proof}
    This is a standard Bertini-type statement, but there is no
    reference to handle the present stacky context.

    Choose $N$ large enough that the following restriction maps are surjective:
    \begin{enumerate}
    \item $\hom_{\ms X}(V,W(N))\to\hom_Z(V|_{Z},W(N)|_Z)$ is
    surjective for every closed substack $Z\subset\ms X$ of the form
    $\spec\ms O_{\ms X}/\ms I_{\xi}^3$, where $\xi\subset\ms X$ is a
    closed residual gerbe;
  \item $\hom_{\ms X}(V,W(N))\to\hom_{\ms D}(V|_{\ms
      D},W(N)|_{\ms D})$.
    \end{enumerate}
   Let $A$ denote the affine space whose
    $k$-points are $\hom_{\ms X}(V,W(N))$ and let
    $\Phi:V|_{\ms X\times A}\to W(N)|_{\ms X\times A}$ be
    the universal map; call the cokernel $\ms N$.  The right-exactness of
    base change and the usual openness results show that there is an
    open subscheme $A^\circ\subset A$ over which $\ms N$ is an
    invertible sheaf over a smooth $A^\circ$-stack.  Our goal is to
    show that $A^\circ$ is non-empty.

    Let $\ms Y\subset\ms X\times A$ denote the open locus over which
    $\ms N$ has geometric fibers of dimension at most $1$ and smooth
    support.  The complement of $\ms Y$ is a closed cone over $\ms X$,
    and we will show that it has codimension at least $3$ in every
    fiber over a closed residual gerbe $\xi$ of $\ms X$ distinct from
    the singular gerbes of $\ms D$.  Since $\ms X$ has dimension $2$,
    this shows that the complement of $\ms Y$ cannot dominate $A$.

    Since $\hom(V,W(N))\to\hom(V|_{Z_\xi}, W(N)|_{Z_\xi})$ is surjective, it suffices to prove the
    statement for the latter, so that we can trivialize the gerbe $\ms
    X$ and thus view $V$ and $W$ as either sheaves over
    $k[x,y]/(x,y)^2$ or as representations of $\m_\ell$ over
    $k[x,y]/(x,y)^2$.  Since $V$ and $W$ are regular, in
    the latter case we have that $V$ and $W$ are both
    $\ell$ times the regular representation.  In either case (and after passing to eigensheaves if necessary), it
    suffices to prove the following.

    \begin{claim}
      Given a free module of rank $n\geq 2$ over $R:=k[x,y]/(x,y)^2$, the
      locus of maps $f\in\M_n(R)$ such that $\det f=0$ or $\dim\coker
      f\tensor k>1$ has $k$-codimension at least $3$ in $\M_n(R)$ (viewed as a
      $k$-vector space).
    \end{claim}

    To see that this suffices, note that if $f:V\to W(N)$
    is a map which avoids the cone of the claim at every point of $\ms
    X$ then $\coker f$ is a sheaf supported on a smooth curve $C$ such
    that for every closed residual gerbe the geometric fiber of
    $\coker f$ has dimension $1$.  It follows from Nakayama's lemma
    that $\coker f$ is an invertible sheaf on $C$.

    \begin{proof}[Proof of Claim]
      Write an element of $\M_n(R)$ as $A=A_0+xA_1+yA_2$.  It is
      well-known that the locus of matrices $A_0$ of rank at most
      $n-2$ has codimension $3$ in $\M_n(k)$ (see, e.g., Lemma
      8.1.9(ii) of \cite{artin-dejong}), settling the second condition.  

      For the first, recall the Jacobi formula
$$\det A=\det A_0+\Tr(\adj(A_0)(xA_1+yA_2)).$$
If $\det A_0=0$ but $A_0\neq 0$, then the condition $\det A=0$ has
codimension $3$, as the vanishing of
$\Tr(\adj(A_0)A_1)$ and $\Tr(\adj(A_0)A_2)$ are independent
conditions.  On the other hand, $A_0=0$ is a codimension at least $3$ condition
as $n\geq 2$.
\end{proof}
As a consequence of the claim, we see that the locus of sections
$Y\subset A$ parametrizing maps $V\to W(N)$ whose cokernel
is not an invertible twisted sheaf supported on a smooth curve is a
proper subvariety of $A$.  Applying the same argument to $\ms D$ shows
that a general point of $A$ parametrizes a map whose cokernel has
support intersecting $\ms D$ transversely, as desired.
  \end{proof}

\section*{Acknowledgments}
\label{sec:acknowledgments}

During the course of this work, the author had helpful conversations
with Dan Abramovich, Jean-Louis Colliot-Th\'el\`ene, Brian Conrad, Aise Johan de
Jong, J\'anos Koll\'ar, Davesh Maulik, Martin Olsson, and Jason Starr.


\begin{thebibliography}{10}

\bibitem{MR0354655}
{\em Th\'eorie des intersections et th\'eor\`eme de {R}iemann-{R}och}.
\newblock Lecture Notes in Mathematics, Vol. 225. Springer-Verlag, Berlin,
  1971.
\newblock S{\'e}minaire de G{\'e}om{\'e}trie Alg{\'e}brique du Bois-Marie
  1966--1967 (SGA 6), Dirig{\'e} par P. Berthelot, A. Grothendieck et L.
  Illusie. Avec la collaboration de D. Ferrand, J. P. Jouanolou, O. Jussila, S.
  Kleiman, M. Raynaud et J. P. Serre.

\bibitem{starr}
Xuhua~He, Aise~Johan~de Jong, and Jason~Starr.
\newblock Families of rationally simply connected varieties over surfaces and
  torsors for semisimple groups.

\bibitem{MR0260746}
Michael~Artin.
\newblock Algebraization of formal moduli. {I}.
\newblock In {\em Global {A}nalysis ({P}apers in {H}onor of {K}. {K}odaira)},
  pages 21--71. Univ. Tokyo Press, Tokyo, 1969.

\bibitem{MR657428}
Michael~Artin.
\newblock Local structure of maximal orders on surfaces.
\newblock In {\em Brauer groups in ring theory and algebraic geometry
  ({W}ilrijk, 1981)}, volume 917 of {\em Lecture Notes in Math.}, pages
  146--181. Springer, Berlin, 1982.

\bibitem{artin-dejong}
Michael~Artin and Aise~Johan de~Jong.
\newblock Stable orders over surfaces, 2003.
\newblock Preprint.

\bibitem{MR0321934}
Michael~Artin and David~Mumford.
\newblock Some elementary examples of unirational varieties which are not
  rational.
\newblock {\em Proc. London Math. Soc. (3)}, 25:75--95, 1972.

\bibitem{ax}
James~Ax. 
\newblock The Elementary Theory of Finite Fields. 
\newblock {\em Annals of Mathematics}, 88 (2): 239--271

\bibitem{cadman_using_2007}
Charles~Cadman.
\newblock Using stacks to impose tangency conditions on curves.
\newblock {\em American Journal of Mathematics}, 129(2):405--427, 2007.

\bibitem{colliot}
Jean-Louis Colliot-Th\'el\`ene.
\newblock Die {B}rauersche {G}ruppe; ihre {V}erallgemeinerungen und
  {A}nwendungen in der {A}rithmetischen {G}eometrie ({V}ortragsnotizen,
  {B}rauer {T}agung, {S}tuttgart, 22.-24. {M}{\"a}rz 2001), 2001.
\newblock Unpublished notes.

\bibitem{MR1923420}
Jean-Louis Colliot-Th{\'e}l{\`e}ne.
\newblock Exposant et indice d'alg\`ebres simples centrales non ramifi\'ees.
\newblock {\em Enseign. Math. (2)}, 48(1-2):127--146, 2002.
\newblock With an appendix by Ofer Gabber.

\bibitem{dejong-gabber}
Aise~Johan de~Jong.
\newblock A result of {G}abber, 2003.
\newblock Preprint.

\bibitem{MR2060023}
Aise~Johan de~Jong.
\newblock The period-index problem for the {B}rauer group of an algebraic
  surface.
\newblock {\em Duke Math. J.}, 123(1):71--94, 2004.

\bibitem{kresch-cycle}
Andrew Kresch.
\newblock Cycle groups for {A}rtin stacks.
\newblock {\em Invent. Math.}, 138(3):495--536, 1999.

\bibitem{MR2026412}
Andrew Kresch and Angelo Vistoli.
\newblock On coverings of {D}eligne-{M}umford stacks and surjectivity of the
  {B}rauer map.
\newblock {\em Bull. London Math. Soc.}, 36(2):188--192, 2004.

\bibitem{orbi}
Max Lieblich.
\newblock Moduli of twisted orbifold sheaves.
\newblock {\em Adv. Math}, to appear.

\bibitem{paiitbgoaas}
Max Lieblich.
\newblock Period and index in the {B}rauer group of an arithmetic surface (with
  an appendix by {D}aniel {K}rashen).
\newblock {\em Crelle}, to appear.

\bibitem{twisted-moduli}
Max Lieblich.
\newblock Moduli of twisted sheaves.
\newblock {\em Duke Mathematical Journal}, 138(1):23--118, 2007.

\bibitem{period-index-paper}
Max Lieblich.
\newblock Twisted sheaves and the period-index problem.
\newblock {\em Compositio Mathematica}, 144(1):1--31, 2008.

\bibitem{MR1432058}
Laurent Moret-Bailly.
\newblock Un probl\`eme de descente.
\newblock {\em Bull. Soc. Math. France}, 124(4):559--585, 1996.

\bibitem{MR1376250}
Kieran~G. O'Grady.
\newblock Moduli of vector bundles on projective surfaces: some basic results.
\newblock {\em Invent. Math.}, 123(1):141--207, 1996.

\bibitem{MR2183251}
Martin~C. Olsson.
\newblock On proper coverings of {A}rtin stacks.
\newblock {\em Adv. Math.}, 198(1):93--106, 2005.

\bibitem{MR1462850}
David~J. Saltman.
\newblock Division algebras over $p$-adic curves.
\newblock {\em J. Ramanujan Math. Soc.}, 12(1):25--47, 1997.

\bibitem{saltman_cyclic_2007}
David~J Saltman.
\newblock Cyclic algebras over $p$-adic curves.
\newblock {\em Journal of Algebra}, 314(2):817--843, 2007.

\bibitem{MR554237}
Jean-Pierre Serre.
\newblock {\em Local fields}, volume~67 of {\em Graduate Texts in Mathematics}.
\newblock Springer-Verlag, New York, 1979.
\newblock Translated from the French by Marvin Jay Greenberg.

\end{thebibliography}
\end{document}